\documentclass[12pt]{article}
\usepackage{amsmath,amssymb,amsthm,a4wide}
\input xypic
\usepackage[all,tips]{xy}

\begin{document}

\newcommand{\Q}{{\bf Q}}
\newcommand{\N}{{\bf N}}
\newcommand{\E}{{\mathcal E}}
\newcommand{\cB}{{\mathcal B}}
\newcommand{\caL}{{\mathcal L}}
\newcommand{\cA}{{\mathcal A}}
\newcommand{\cC}{{\mathcal C}}
\newcommand{\cP}{{\mathcal P}}
\newcommand{\cS}{{\mathcal S}}
\newcommand{\cZ}{{\mathcal Z}}
\newcommand{\caD}{{\mathcal D}}
\newcommand{\Z}{{\bf Z}}
\newcommand{\R}{{\bf R}}
\newcommand{\C}{{\bf C}}
\newcommand{\M}{{\bf M}}
\newcommand{\bP}{{\bf P}}
\newcommand{\F}{{\bf F}}
\newcommand{\cG}{{\mathcal G}}
\newcommand{\caH}{{\mathcal H}}
\newcommand{\cO}{{\mathcal O}}
\newcommand{\AS}{\mathcal{AS}}
\newcommand{\SAS}{\mathcal{SAS}}
\newcommand{\cM}{{\mathcal M}}
\newcommand{\cN}{{\mathcal N}}
\newcommand{\se}{:}
\newcommand{\gl}{\mathfrak{gl}}
\newcommand{\sve}{{\scriptscriptstyle {\vee}}}

\newcommand{\ts}{\,}
\newcommand{\tss}{\hspace{1pt}}
\newcommand{\hra}{\hookrightarrow}
\newcommand{\wt}{\widetilde}
\newcommand{\wh}{\widehat}
\newcommand{\ot}{\otimes}
\newcommand{\Hom}{{\rm Hom}}
\newcommand{\End}{{\rm End}}
\newcommand{\Aut}{{\rm Aut}}
\newcommand{\Rep}{{\mathcal Rep}}
\newcommand{\Vect}{{\mathcal Vect}}
\newcommand{\Mod}{{\mathcal Mod}}
\newcommand{\Mor}{{\mathcal Mor}}
\newcommand{\Funct}{{\mathcal Funct}}
\newcommand{\da}{{\mbox{-}}}
\newcommand{\dda}{{\rm{-}}}
\newcommand{\U}{ {\rm U}}
\newcommand{\Y}{ {\rm Y}}
\newcommand{\He}{ {\rm H}}
\newcommand{\non}{\nonumber}

\renewcommand{\theequation}{\arabic{section}.\arabic{equation}}

\newtheorem{thm}{Theorem}[section]
\newtheorem{lem}[thm]{Lemma}
\newtheorem{prop}[thm]{Proposition}
\newtheorem{cor}[thm]{Corollary}
\newtheorem{conj}[thm]{Conjecture}

\theoremstyle{definition}
\newtheorem{defin}[thm]{Definition}

\theoremstyle{remark}
\newtheorem{remark}[thm]{Remark}
\newtheorem{example}[thm]{Example}

\newcommand{\bth}{\begin{thm}}
\renewcommand{\eth}{\end{thm}}
\newcommand{\bpr}{\begin{prop}}
\newcommand{\epr}{\end{prop}}
\newcommand{\ble}{\begin{lem}}
\newcommand{\ele}{\end{lem}}
\newcommand{\bco}{\begin{cor}}
\newcommand{\eco}{\end{cor}}
\newcommand{\bde}{\begin{defin}}
\newcommand{\ede}{\end{defin}}
\newcommand{\bex}{\begin{example}}
\newcommand{\eex}{\end{example}}
\newcommand{\bre}{\begin{remark}}
\newcommand{\ere}{\end{remark}}
\newcommand{\bcj}{\begin{conj}}
\newcommand{\ecj}{\end{conj}}

\newcommand{\bal}{\begin{aligned}}
\newcommand{\eal}{\end{aligned}}
\newcommand{\beq}{\begin{equation}}
\newcommand{\eeq}{\end{equation}}
\newcommand{\ben}{\begin{equation*}}
\newcommand{\een}{\end{equation*}}

\newcommand{\bpf}{\begin{proof}}
\newcommand{\epf}{\end{proof}}

\def\beql#1{\begin{equation}\label{#1}}

\title{\Large\bf A categorical approach to classical
and quantum Schur--Weyl duality}

\author{Alexei Davydov\quad and\quad Alexander Molev}

\date{} 

\maketitle

\vspace{25 mm}

\begin{abstract}
We use category theory
to propose a unified approach to
the Schur--Weyl dualities involving the general linear
Lie algebras, their polynomial extensions and
associated quantum deformations.
We define multiplicative sequences
of algebras exemplified by the sequence of group algebras
of the symmetric groups and use them to introduce
associated monoidal categories. Universal properties
of these categories lead to uniform constructions
of the Drinfeld functor connecting representation theories
of the degenerate affine Hecke algebras and the Yangians
and of its $q$-analogue. Moreover, we construct
actions of these categories on certain
(infinitesimal) braided categories containing a Hecke object.
\end{abstract}

\vspace{25 mm}

\noindent
Max Planck Institut f\"ur Mathematik\\
Vivatsgasse 7, 53111 Bonn, Germany\\
alexei1davydov@gmail.com

\vspace{7 mm}

\noindent
School of Mathematics and Statistics\newline
University of Sydney,
NSW 2006, Australia\newline
alexander.molev@sydney.edu.au

\newpage

\tableofcontents

\newpage

\section{Introduction}\label{sec:int}
\setcounter{equation}{0}

Classically, starting with the vector representation $V$ of $\gl_N$
and taking its tensor powers $V^{\otimes n}$
one gets a sequence of the corresponding endomorphism algebras.
In the limiting case
$N\to\infty$, the sequence of the endomorphism algebras stabilizes
and admits a precise description: this is a sequence of
the group algebras of the symmetric groups $S_n$.

In a quantum version of the Schur--Weyl duality, the Lie algebra
$\gl_N$ is replaced by the quantized enveloping algebra $\U_q(\gl_N)$,
and the corresponding stable sequence of the endomorphism algebras
is a sequence of the Hecke algebras $\He_n(q)$.
Similarly, the endomorphism algebras associated with the
representations of the polynomial current and loop Lie algebras
and their quantizations
lead to sequences of group algebras of
affine extensions of the symmetric groups and the
sequences of affine Hecke algebras
and their degenerate versions.

It is these sequences of algebras which we take as a starting point
of our approach. We use them to construct certain categories
(which we call the {\it Schur--Weyl categories\/}) modeling the
parts of the corresponding representation categories generated
by the vector representations. Each of the Schur--Weyl categories
possesses a {\it universal property\/}: as a monoidal category,
it is generated by a single object
and a collection of endomorphisms. This property leads to
a simple characterization of
monoidal functors from these categories thus allowing us to reformulate
the Schur--Weyl dualities as the existence and fullness properties
of functors from the Schur--Weyl categories to the appropriate
representation categories.

As another application of the universal property, we
show that a localized category associated with the
(degenerate) affine Hecke algebras is equivalent
to a localized category associated with the
(semi-) affine symmetric group algebras.
Moreover, motivated by the work \cite{or:ab}
we construct actions of the Schur--Weyl categories
corresponding to
the (degenerate) affine Hecke algebras on certain
(infinitesimal) braided
categories containing Hecke objects.

Category theory approaches to the Schur--Weyl
duality were developed by many authors; see e.g. \cite{bk:sw},
\cite{bs:hw3}, \cite{bs:hw4}, \cite{cr:de}, \cite{ms:ci} and
a recent review \cite{s:sw}. In our interpretation
of the Schur--Weyl duality we do not ``categorify" it in the sense
of \cite{s:sw}, but rather give a categorical viewpoint
based on the universality properties of the Schur--Weyl categories.
Note also that the ``higher Schur--Weyl duality" involving
cyclotomic Hecke algebras does not appear to fit into our context as
these algebras do not form multiplicative sequences;
cf. \cite{bk:sw} and \cite{ms:ci}.

We will now explain our approach in more detail by using the
classical Schur--Weyl duality as a model example.
One consequence of this duality is the fact that there are
no non-zero homomorphisms between the $m$-th and $n$-th
tensor powers of the $N$-dimensional vector
representation $V$ of the general linear Lie algebra $\gl_N$
for $m\ne n$:
\ben
\Hom_{\gl_N}(V^{\otimes m},V^{\otimes n})=0.
\een
Moreover, the natural homomorphism
\ben
k[S_n]\to \End_{\gl_N}(V^{\otimes n})
\een
from the group algebra of the symmetric group is surjective
(it is also injective if $n\leqslant N$)
\cite{h:pit}, \cite{sh:ud} and \cite{w:cg}. Here and
throughout the paper $k$
denotes a field, and vector spaces and algebras are considered
over $k$, unless stated otherwise.

At this point we want to shift the emphasis from
the Lie algebra $\gl_N$ to the
sequence of the symmetric group algebras
\ben
k[S_*]=\{k[S_n]\ |\ n\geqslant 0\}
\een
and to regard this sequence as
the primary object of the duality.
The sequence $k[S_*]$ comes equipped with the
algebra homomorphisms
\beql{hommu}
\mu_{m,n}:k[S_m]\otimes k[S_n]\to k[S_{m+n}],
\eeq
induced by the natural inclusions $S_m\times S_n\hra S_{m+n}$.
The homomorphisms $\mu_{m,n}$ satisfy a certain associativity
property which we describe below in Sec.\ts\ref{sec:msa}.
The collection $k[S_*]$ together with these homomorphisms
is an example of what we call a {\it multiplicative
sequence of algebras\/}.

We can use the multiplicative sequence $k[S_*]$ to define a $k$-linear
monoidal category
$\overline\cS$ whose objects $[n]$ are parameterized by
the natural numbers $n\geqslant 0$
with no morphisms between the objects corresponding to
different numbers and with the endomorphism algebra
$\End_{\overline\cS}([n])$
of the $n$-th object
being $k[S_n]$. The tensor product on objects is defined by
the addition of natural numbers, $[m]\otimes [n]=[m+n]$, while
on the morphisms it is defined by the algebra homomorphisms
\eqref{hommu}. The category
$\overline\cS$ possesses a universal property: it is a free symmetric
monoidal $k$-linear category generated by one object $X=[1]$.

We denote by $\cS$ the category
of $k$-linear functors from
the opposite category $\overline\cS^{\ts op}$ to the category of
vector spaces. Then $\cS$ is a free symmetric abelian
monoidal $k$-linear category generated by one object.
More explicitly, $\cS$
is the direct sum of
the categories of right modules $\oplus_n \Mod\dda k[S_n]$.
The classical Schur--Weyl duality implies that the symmetric monoidal
functor from $\cS$ to the category of representations of $\gl_N$
sending the generator $X$ to $V$ is full (surjective on morphisms).
In other words, the universal enveloping algebra of $\gl_N$ is
Tannaka--Krein dual to the symmetric monoidal functor from $\cS$ to
the category of vector spaces, sending the generator $X$ to $V$.

This construction can be generalized as follows. Let $A$
be a unital associative algebra. One can start with a free symmetric
monoidal category $\cS(A)$
generated by one object $X$ and with $A$
as the algebra of its endomorphisms. The
corresponding multiplicative sequence of algebras is now the
sequence of cross-products $A^{\otimes n}*S_n$.
We can associate with an $N$-dimensional vector
space $V$ the symmetric monoidal functor from
$\cS(A)$ to the category of vector spaces which sends the generator $X$
to $V\otimes A$, considered as a vector space. Its Tannaka--Krein dual
is the universal enveloping algebra of the Lie algebra
$\End_A(V\otimes A)$ of $A$-linear endomorphisms of $V\otimes A$. Two
examples of this situation will be particularly important for us due
to their interesting quantizations. Namely, these are
the cases where $A$ is the
algebra of polynomials in one variable, or its Laurent version. Then
the Tannaka--Krein duals are respectively
the universal enveloping algebras of the polynomial
current Lie algebra $\gl_N[t]$ and of the Lie algebra
of Laurent polynomials $\gl_N[t,t^{-1}]$.
The corresponding (Schur--Weyl dual)
multiplicative sequences of algebras are the sequences of group
algebras of affine symmetric (semi-)groups.

More examples come via the Schur--Weyl duality
from quantizations of the general linear Lie algebras or their current
deformations. Namely, the multiplicative sequence of Hecke algebras
arises from the endomorphism algebras of the tensor powers of the vector
representation of the quantized enveloping algebra $\U_q(\gl_N)$.
This gives rise to a monoidal category which we call the
{\it Hecke category\/}. Furthermore, the quantum
loop algebras and Yangians are respective
quantizations of the algebras
$\U(\gl_N[t,t^{-1}])$ and $\U(\gl_N[t])$.
The Schur--Weyl dual sequence associated with the
quantum loop algebras is formed by the affine Hecke algebras, while
the dual sequence for the Yangians
is formed by the sequence of degenerate affine Hecke algebras.

Now we outline the contents of the paper.
We start by
defining multiplicative sequences of algebras and the corresponding
Schur--Weyl categories and list their basic properties
(Sec.\ts\ref{sec:msa}). In particular, we analyze a condition on an
embedding of multiplicative sequences which guarantees that the
right adjoint to the monoidal functor induced by the embedding
(the restriction functor along the embedding) is
also monoidal. This will be used later to construct fiber functors
(i.e. monoidal functors to vector spaces) for
the Schur--Weyl categories of (degenerate) affine Hecke algebras.
We also study presentations of multiplicative sequences in terms of
generators and relations and corresponding freeness properties of
the Schur--Weyl categories.

In the subsequent sections we consider the families of
multiplicative sequences of algebras which are Schur--Weyl
dual to the general linear Lie algebra,
its polynomial current versions and their quantum
deformations: the quantized
enveloping algebra, the Yangian and the quantum loop algebra.
One of the advantages of our approach is a natural
and unifying construction of the actions of the Hecke algebras (or their
affine and degenerate versions) on tensor powers of the
corresponding (polynomially extended) vector representation.
In the case of the Yangians this action does not appear to have been
previously described in the literature; cf. \cite{a:df}, \cite{d:da}.
In the quantum loop algebra case, the action of
the affine Hecke algebras
on the tensor powers of the vector
representation provides an alternative to the construction
previously given in \cite[Theorem~4.9]{grv:qg}.

In the last section we prove equivalences of certain localized
categories and construct categorical actions
of the Schur--Weyl categories.

\medskip

The work on the paper began while the first author was
visiting the University of Sydney and it was completed during his
visit to Max Planck Institut f\"ur Mathematik (Bonn) in
2010. He would like to thank these institutions for hospitality
and excellent working conditions. We would also like
to thank the Australian Research Council and Max Planck Gesellschaft
for supporting these visits.

\section{Sequences of algebras and monoidal categories}
\label{sec:msa}
\setcounter{equation}{0}

Throughout the paper we will freely use the standard language
of categories and functors; see e.g. \cite{m:cw}.
All categories and functors will be
linear over the ground field $k$, unless stated otherwise.
All monoidal functors will be understood as strong monoidal.
The set (or vector space) of
morphisms between objects $X,Y$
of a category $\cC$ will be
denoted by $\cC(X,Y)$ and the endomorphism algebra
$\cC(X,X)$ of $X\in\cC$ will
be denoted by $\End_\cC(X)$.

\subsection{Multiplicative sequences of algebras and monoidal
categories}\label{subsec:ms}

We will deal with sequences of associative unital algebras
$A_*=\{A_n\ |\ n\geqslant 0\}$ equipped with collections
of (unital) algebra
homomorphisms
$$\mu_{m,n}:A_m\otimes A_n\to A_{m+n},\qquad m,n\geqslant 0,$$
satisfying the following associativity
axiom: for any $l,m,n\geqslant 0$ the following
diagram commutes:
$$
\xymatrix{ A_l\otimes A_m\otimes A_n
\ar[rr]^{\mu_{\tss l,m}\ot I} \ar[d]_{I\ot\mu_{m,n}} &&
A_{l+m}\otimes A_n \ar[d]^{\mu_{\tss l+m,n}} \\
A_l\otimes A_{m+n} \ar[rr]^{\mu_{\tss l,m+n}} && A_{l+m+n} .}
$$
We call such a sequence {\it multiplicative} and we will
always assume that $A_0=k$.
A model example
of a multiplicative sequence of algebras is provided by the following
construction.
Let $\cC$ be a (strict) monoidal category
such that $\End_\cC(I)=k$,
where $I$ denotes the unit object of $\cC$.
Given an object $X$ of
$\cC$, the sequence $A_*$ with $A_n=\End_\cC(X^{\otimes n})$
is multiplicative
with respect to the homomorphisms $\mu_{m,n}$ given by the tensor
product on morphisms

$$
\End_\cC(X^{\otimes m})\otimes \End_\cC(X^{\otimes n})
\to \End_\cC(X^{\otimes m+n}).
$$
Moreover, any multiplicative sequence can be
obtained in this way. Indeed,
starting with a multiplicative sequence $A_*$, define the category
$\overline\cC(A_*)$ with objects $[n]$ parameterized by natural numbers,
with no morphisms between different objects and with the endomorphism
algebras $\End_{\overline\cC(A_*)}([n])=A_n$. Define tensor product
on the objects of $\overline\cC(A_*)$ by $[m]\otimes [n]=[m+n]$.
The multiplicative structure of the sequence $A_*$ yields the tensor
product on morphisms.

Let $f_*=\{f_n\ |\ n\geqslant 0\}$ be a sequence of
algebra homomorphisms
$f_n:A_n\to B_n$ between the corresponding algebras of two
multiplicative sequences $A_*$ and $B_*$.
We call $f_*$ a {\it homomorphism}
of multiplicative sequences if for any $m,n$ the following diagram
commutes:
$$\xymatrix{ A_m\otimes A_n \ar[rr]^{f_m\otimes f_n}
\ar[d]_{\mu_{m,n}} && B_m\otimes B_n \ar[d]^{\mu_{m,n}} \\
A_{m+n} \ar[rr]^{f_{m+n}} && B_{m+n} .}
$$
We will say that $f_*$ is an {\it epimorphism\/}, if all
homomorphisms $f_n$ are surjective.

Note that the construction of the category
$\overline\cC(A_*)$ is functorial with respect to homomorphisms of
multiplicative sequences: a homomorphism $f_*:A_*\to B_*$
defines a monoidal functor
$\overline\cC(f_*):\overline\cC(A_*)\to\overline\cC(B_*)$.

Now we denote by $\cC(A_*)$ the category
$\Funct\big(\overline\cC(A_*)^{op},\Vect\big)$
of $k$-linear functors from
the opposite category $\overline\cC(A_*)^{op}$ to the category of
vector spaces over $k$. As a category, $\cC(A_*)$ is the direct sum of
categories of right modules $\oplus_n \Mod\dda A_n$. The monoidal
structure on $\cC(A_*)$ (induced by the monoidal structure on
$\overline\cC(A_*)$ via Day's convolution) is given by
\beql{monoi}
M\ot_{A_*} N
= (M\otimes N)\otimes_{A_m\otimes A_n}A_{m+n}.
\eeq
Here $M$ is an
$A_m$-module, $N$ is an $A_n$-module, and $A_{m+n}$ is considered as
a left $A_m\otimes A_n$-module via the homomorphism $\mu_{m,n}$. We
call the monoidal category $\cC(A_*)$ the {\it Schur--Weyl category}
corresponding to the multiplicative sequence of algebras $A_*$.
The subscript $A_*$ of the tensor product sign in \eqref{monoi}
will usually be omitted, if no confusion is possible as to what
multiplicative sequence is used or what monoidal category
is considered.

Suppose that $\overline\cC$ is a $k$-linear category and $\cC$ is an
abelian category together with a fully-faithful $k$-linear functor
$F:\overline\cC\to\cC$. We call $\cC$ an {\it abelian envelope} of
$\overline\cC$ if any object of $\cC$ is a subquotient of a direct
sum of objects from the image of $F$.

\begin{lem}
The category  $\cC(A_*)$ is an abelian envelope of
$\overline\cC(A_*)$.
\end{lem}

\begin{proof}
By its definition, the category $\cC(A_*)$ is abelian. The category
$\overline\cC(A_*)$ is fully and faithfully embedded into
$\overline\cC(A_*)$ via the Yoneda functor $Y\mapsto
\overline\cC(A_*)(-,Y)$. Explicitly, the image of $[n]$ via this
embedding is $A_n$ considered as a module over itself. It is clear
that any
object in $\Mod\dda A_n$ is a quotient of a direct sum of copies of
$A_n$.
\end{proof}

Note that the construction of $\cC(A_*)$ is functorial with respect to
homomorphisms of multiplicative sequences: a homomorphism
$f_*:A_*\to B_*$ of multiplicative sequences defines a functor
$$
\cC(f_*):\cC(A_*)\to\cC(B_*),\qquad \cC(f_*)(M)
= M\otimes_{A_m}B_m,
$$
if $M$ is an $A_m$-module.
Moreover, this functor is monoidal:
$$\xymatrix{\cC(f_*)(M\ot_{A_*} N) \ar@{=}[r]
&(M\ot_{A_*} N)\otimes_{A_{m+n}}B_{m+n}\ar[d]\\ &(M\otimes
N)\otimes_{A_m\otimes A_n}B_{m+n} \ar[d] \\ \cC(f_*)(M)\ot_{B_*}
\cC(f_*)(N) \ar@{=}[r] &\big((M\otimes_{A_m}
B_m)\otimes(N\otimes_{A_n}B_n)\big)\otimes_{B_m\otimes B_n}B_{m+n}.}
$$

The following results will be important for proving monoidality of the
right adjoint of $\cC(f_*)$. We start with a
technical definition. We will say that an algebra $A$ admits a {\it
multiplicative decomposition} $A=BC$,
if $B$ and $C$ are
subalgebras of $A$ and any element of $A$ can be uniquely written as
a product of elements of $B$ and $C$, i.e. the multiplication in $A$
induces the isomorphism of vector spaces $B\otimes C\to A$.

\begin{lem}\label{tmd}
Let $A$ be an algebra with multiplicative decomposition $A=BC$ and
let $A'\subset A$ be a subalgebra with multiplicative decomposition
$A'=BC'$, where $C'\subset C$ is a subalgebra. Then for any
$A'$-module $M$ the natural homomorphism
\begin{equation}\label{ct}
M\otimes_{C'}C\to M\otimes_{A'}A,\qquad m\otimes c\mapsto m\otimes c,
\end{equation}
induced by the embedding
$C\subset A$, is an isomorphism of $C$-modules.
\end{lem}

\begin{proof}
Define the inverse to the homomorphism (\ref{ct}) as follows. For
$a\in A$ write $a=bc$ for unique $b\in B$ and $c\in C$.
Define the image
of the element
$m\otimes a =
m\otimes bc\in M\otimes_{A'}A$ to be
$mb\otimes c\in M\otimes_{C'}C$.
It is straightforward to check that this is a well-defined map,
which is inverse to the map  (\ref{ct}).
\end{proof}

The next theorem gives a sufficient condition which guarantees
that the restriction functor along an embedding of multiplicative
sequences of algebras is monoidal.

\begin{thm}\label{aff}
Suppose that $C_n\subset A_n$, $n=0,1,\dots$, is an embedding
of multiplicative sequences of
algebras satisfying the condition that for
each $n$ the algebra $A_n$ admits a multiplicative
decomposition
$A_n=A^{\otimes n}C_n$ for a certain algebra $A$. Then the restriction
functor $\cC(A_*)\to \cC(C_*)$ along the embedding $C_*\subset A_*$
is monoidal.
\end{thm}

\begin{proof}
Note that the algebra $A_m\otimes A_n$ has a multiplicative
decomposition
$$
A_m\otimes A_n = (A^{\otimes m}C_m)\otimes(A^{\otimes n}C_n) =
A^{\otimes(m+n)}(C_m\otimes C_n).
$$
Thus, by Lemma~\ref{tmd} the natural morphism
$$
M\ot_{C_*}N = (M\otimes N) \otimes_{C_m\otimes C_n}C_{m+n}\to
(M\otimes N)\otimes_{A_m\otimes A_n}A_{m+n} = M\ot_{A_*}N
$$
is an isomorphism.
The coherence condition for this isomorphism is straightforward
from its definition; cf. \cite[Ch.~VII]{m:cw}.
\end{proof}

\subsection{Generators and relations}

We start by defining free multiplicative sequences of algebras.
Let $V = \{V_l\ |\ l\geqslant 1\}$ be a
collection of vector spaces $V_l$
parameterized by positive integers.
Define another collection of vector
spaces $\{M(V)_m \ |\ m\geqslant 1\}$ by
$$
M(V)_m = \bigoplus_{l=1}^m\Big(\bigoplus_{i=0}^{m-l}V_l(i)\Big),
$$
where $V_l(i) = \{v(i) \ |\ v(i)\in V_l\}$ is a copy of the space
$V_l$ labeled by the index $i$.
The components $M(V)_m$ are connected by linear maps
$$\alpha_{m,n}:M(V)_m\oplus M(V)_n\to M(V)_{m+n},$$
where for any $1\leqslant l\leqslant m,\ 1\leqslant k\leqslant n$ and
$0\leqslant i\leqslant m-l,\ 0\leqslant j\leqslant n-k$ we have
$$
\alpha_{m,n}(v(i),w(j)) = v(i) + w(m+j),\qquad v\in V_l,\quad w\in V_k.
$$
By the definition of these maps, the following diagram commutes:
$$
\xymatrix{ M(V)_l\oplus M(V)_m\oplus
M(V)_n \ar[rr]^{\quad\alpha_{l,m}\oplus I}
\ar[d]_{I\oplus\alpha_{m,n}} &&
M(V)_{l+m}\oplus M(V)_n \ar[d]^{\alpha_{l+m,n}} \\
M(V)_l\oplus M(V)_{m+n} \ar[rr]^{\alpha_{l,m+n}} && M(V)_{l+m+n} .}
$$
Now let $T\big(M(V)_m\big)$ denote the tensor algebra
of the vector space $M(V)_m$.
Denote by
$J_m$ the two-sided ideal of $T\big(M(V)_m\big)$ generated by
all elements of the form
$v(i)w(j)-w(j)v(i)$, where $v\in V_l,\ w\in V_k$
with $1\leqslant l\leqslant m$, $1\leqslant k\leqslant n$,
and
the indices $i$ and $j$ satisfy the conditions
$0\leqslant i\leqslant m-l$, $0\leqslant j\leqslant n-k$ and
\ben
\{i+1,i+2,\dots,i+l\}\cap \{j+1,j+2,\dots,j+k\}=\varnothing.
\een
For $m\geqslant 1$ denote by $A(V)_m$ the quotient algebra
$T\big(M(V)_m\big)/J_m$ and set $A(V)_0=k$.
By the definition of the maps $\alpha_{m,n}$, the difference
$$\alpha_{m,n}(x,0) - \alpha_{m,n}(0,y)$$
belongs to $J_{m+n}$
for any $x\in M(V)_m$ and $y\in M(V)_n$.
Hence, the natural homomorphisms
$$T(\alpha_{m,n}):T\big(M(V)_m\oplus
M(V)_n\big)\to T\big(M(V)_{m+n}\big)/J_{m+n}$$
induced by the linear
maps $\alpha_{m,n}$ factor through
$T\big(M(V)_m\big)\otimes T\big(M(V)_n\big)$.
Moreover, the kernel of the homomorphism
$$T\big(M(V)_m\big)\otimes T\big(M(V)_n\big)\to T
\big(M(V)_{m+n}\big)/J_{m+n}$$
contains $J_m\otimes 1 + 1\otimes J_n$.
Thus we have the algebra homomorphisms
\ben
\mu_{m,n}:A(V)_m\otimes A(V)_n \to A(V)_{m+n}
\een
which turn
$A(V)_*=\{A(V)_m\ |\ m\geqslant 0\}$ into a
multiplicative sequence.
The sequence $A(V)_*$ is the {\it free multiplicative
sequence of algebras\/} generated by the collection $V$
of vector spaces in the sense that given
any multiplicative sequence of algebras
$A_*=\{A_n\ |\ n\geqslant 0\}$, the homomorphisms
of multiplicative sequences $A(V)_*\to A_*$
are in one-to-one correspondence with the sequences of
vector space homomorphisms $V_l\to A_l$, $l\geqslant 1$.

Now we will use free multiplicative
sequences to give the definition of
a multiplicative sequence of algebras $A_*$ {\it generated by a
family $\{a_i\}_{i\in I}$ of elements} $a_i\in A_{n(i)}$,
$n(i)\geqslant 1$. To this end, we
let $\{\wh a_i\}_{i\in I}$ be a set equipped with
the degree function $\deg\wh a_i=n(i)$.
We will use this set to span a collection $V$ of
vector spaces over the field $k$ by setting
\beql{colvs}
V_l = \text{span of}\ \ \{\wh a_i \ |\ \deg\wh a_i=l\}.
\eeq
We will say that a
multiplicative sequence of algebras $A_*$ is generated by a
family $\{a_i\}_{i\in I}$ of elements $a_i\in A_{n(i)}$ if there is an
epimorphism of multiplicative sequences $A(V)_*\to A_*$
such that $\wh a_i\mapsto a_i$ for all $i\in I$.

Consider the collection of vector spaces
$K_*=\{K_n\ |\ n\geqslant 1\}$, where $K_n$ is the kernel
of the epimorphism $A(V)_n\to A_n$.
Suppose that $\{p_j\}_{j\in J}$ is a
set of elements in the kernels, $p_j\in K_{m(j)}$. Let
$P=\{P_l\ |\ l\geqslant 1\}$ be the collection of vector
spaces spanned by this set so that
\ben
P_l = \text{span of}\ \ \{p_j \ |\ \deg p_j=l\}.
\een
We will say that a multiplicative
sequence of algebras $A_*$ is generated by a set
$\{a_i\}_{i\in I}$ of elements $a_i\in A_{n(i)}$ {\it subject to
the relations}
\beql{msrel}
p_j\big(\{a_i\}_{i\in I}\big) = 0,\quad j\in J,
\eeq
if for any $n\geqslant 1$ the
two-sided ideal $K_n$ of $A(V)_n$ is generated by $M(P)_n$.

Furthermore, using free multiplicative sequences
in a similar way, we can now
give the definition of
a {\it monoidal category $\cC$ generated by its
object $X$ and a family $\{a_i\}_{i\in I}$ of endomorphisms
$a_i\in\End_\cC(X^{\otimes n(i)})$ subject to relations}
\beql{catrel}
p_j\big(\{a_i\}_{i\in I}\big) = 0,\quad j\in J.
\eeq
Namely, exactly as above, we
let $\{\wh a_i\}_{i\in I}$ be a set with
the degree function $\deg\wh a_i=n(i)$ and introduce
a collection $V=\{V_l\ |\ l\geqslant 1\}$
of vector spaces by \eqref{colvs}.
The homomorphism of
multiplicative sequences
$A(V)_n\to \End_{\cC}(X^{\otimes n})$, $\wh a_i\mapsto a_i$,
allows us
to think of the elements of $A(V)_*$ as endomorphisms of
the tensor powers of $X$. In particular, it allows us to say that
the endomorphisms $a_i$ satisfy a relation
$p\big(\{a_i\}_{i\in I}\big)=0$, where $p\in A(V)_*$.

We will say that a monoidal category $\cC$ is generated by its
object $X$ and a set $\{a_i\}_{i\in I}$ of endomorphisms $a_i\in
\End_\cC(X^{\otimes n(i)})$ subject to the relations \eqref{catrel},
if for any monoidal category $\caD$ the assignment
$F\mapsto (F(X), F(a_i)_{i\in I})$ defines an equivalence between
the category of monoidal functors $\cC\to\caD$ and
the category whose objects are collections
$(Y,\{b_i\}_{i\in I})$, where $Y$ is an object
of $\caD$ and
$\{b_i\}_{i\in I}$ is a set of endomorphisms
$b_i\in\End_\caD(Y^{\otimes n(i)})$
such that $p_j\big(\{b_i\}_{i\in I}\big) = 0$ for all $j\in J.$

\begin{thm}\label{gr}
Let $A_*$ be a multiplicative sequence of algebras generated by a
set $\{a_i\}_{i\in I}$ of elements $a_i\in A_{n(i)}$ subject to
certain relations \eqref{msrel}.
Then, as a monoidal category, the Schur--Weyl
category $\cC(A_*)$ is generated  by the object
$X = [1]$ and the set $\{a_i\}_{i\in I}$ of endomorphisms $a_i\in
\End_{\cC(A_*)}(X^{\otimes n(i)})$ subject to
the relations \eqref{catrel}.
\end{thm}

\begin{proof}
Suppose that $\caD$
is a monoidal category and consider the category
of monoidal functors $\cC(A_*)\to\caD$. This functor category
is equivalent to the category
with the objects $(Y,g_*)$,
where $Y$ is an object of $\caD$ and
$g_*=\{g_n\ |\ n\geqslant 0\}$
is a homomorphism of multiplicative
sequences
\beql{monh}
g_n:A_n\to \End_\caD(Y^{\otimes n}),\qquad n\geqslant 0.
\eeq
Now if the multiplicative sequence of algebras $A_*$ is generated by
a set $\{a_i\}_{i\in I}$ of elements $a_i\in A_{n(i)}$ subject to
the relations \eqref{msrel}, then the homomorphisms
\eqref{monh} are in a one-to-one correspondence
with the families of endomorphisms $\{b_i\}_{i\in I}$ with
$b_i\in\End_\caD(Y^{\otimes n(i)})$ such that
$p_j\big(\{b_i\}_{i\in I}\big) = 0$ for all $j\in J$.
\end{proof}

\section{Symmetric groups and general linear Lie algebras}\label{sym}
\setcounter{equation}{0}

We start with the simplest case of the classical
Schur--Weyl duality, namely the duality
between the symmetric groups and the general linear Lie algebras.

\subsection{Symmetric groups and their group algebras}

Consider the standard presentation of the group $S_n$
of permutations of the set $\{1,\dots,n\}$ so that
$S_n$ is generated by the elements
$t_1,\dots,t_{n-1}$ subject to the relations:
\beql{cox}
t_i^2=1,\qquad t_it_{i+1}t_i=t_{i+1}t_it_{i+1},
\qquad t_it_j=t_jt_i\quad\text{for}\quad |i-j|>1.
\eeq
The assignments
$$
t_i\otimes 1\mapsto t_i,\qquad 1\otimes t_j\mapsto t_{j+m},
$$
define natural algebra homomorphisms
$$
\mu_{m,n}:k[S_m]\otimes k[S_n]\to k[S_{m+n}]
$$
which satisfy the associativity axiom; see Sec.~\ref{subsec:ms}.
Thus we get a multiplicative sequence of algebras
$k[S_*]=\{k[S_n]\ |\ n\geqslant 0\}$.

\begin{prop}\label{gen}
The multiplicative sequence $k[S_*]$ is generated by
the element $t\in k[S_2]$
subject to the relations $t^2=1$ in $k[S_2]$ and
$$
\mu_{2,1}(t\otimes 1)\mu_{1,2}(1\otimes t)\mu_{2,1}(t\otimes 1) =
\mu_{1,2}(1\otimes t)\mu_{2,1}(t\otimes 1)\mu_{1,2}(1\otimes t)
\qquad
\text{in}\quad k[S_3].
$$
\end{prop}
\begin{proof}
This follows from the presentation of $S_n$; the element
$t=t_1$ is the non-identity element of $S_2$.
\end{proof}

\subsection{The free symmetric category}

As an immediate consequence of Theorem~\ref{gr}
and Proposition~\ref{gen}
we have the following universal property of
the category $\cC(k[S_*])$.
\bth\label{thm:freesym}
The abelian monoidal category $\cS=\cC(k[S_*])$ is generated by an
object $X$ and an automorphism $c:X\otimes X\to X\otimes X$
subject to the relations
\begin{equation}\label{iyb}
c^2=1,\qquad (c\otimes 1_X)(1_X\otimes c)(c\otimes 1_X) = (1_X\otimes
c)(c\otimes 1_X)(1_X\otimes c).
\end{equation}
\eth

We will call
$X$ and $c$ satisfying \eqref{iyb}, an {\it
involutive Yang--Baxter object} and
an {\it involutive Yang--Baxter operator},
respectively.
Theorem~\ref{thm:freesym}
states that $\cS$ is a free abelian monoidal category
generated by an involutive
Yang--Baxter object.

We also formulate another universal property of $\cS$,
which will be used in the next section.
Let $c_{m,n}\in S_{m+n}$ be the $(m,n)$-shuffle preserving the orders
of the first $m$ and last $n$ letters. That is, under this permutation,
\ben
i\mapsto i+n,\quad i=1,\dots,m,\qquad
j\mapsto j-m,\quad j=m+1,\dots,m+n.
\een
The collection $\{c_{m,n}\ |\ m,n\geqslant 0\}$ of the
isomorphisms
$c_{m,n}:X^{\otimes m}\otimes X^{\otimes n}\to X^{\otimes n}\otimes
X^{\otimes m}$ defines a symmetry on the
category $\cS$; see \cite{m:cw}. The following freeness property of
the category $\cS$ is well-known; see \cite{js:bt}.

\begin{prop}\label{prop:free-ab-sym}
The category $\cS$ is a free abelian symmetric monoidal category
generated by one object.
\end{prop}

\subsection{Fiber functors and general linear Lie algebras}
\label{subsec:ffcla}

As before, we let $\Vect$ denote the category of vector spaces
over $k$. A monoidal $k$-linear functor from a certain
category to $\Vect$ will be called a {\it fiber functor\/}.
Due to the freeness property of $\cS$ as monoidal category, fiber
functors $\cS\to\Vect$ correspond to involutive Yang--Baxter
operators. There are quite a number of them forming algebraic moduli
spaces; see \cite{da:tm}. However, symmetric fiber functors are
labeled by vector spaces (the associated Yang--Baxter operator is the
standard permutation of the tensor factors in the tensor square). Thus,
up to monoidal isomorphisms, symmetric fiber functors correspond to
non-negative integers (the dimensions of the corresponding vector
spaces). Denote by $F_N:\cS\to\Vect$ the (isomorphism class of the)
symmetric fiber functor labeled by $N$.

Now the classical Schur--Weyl duality can be interpreted as follows.

\begin{prop}\label{prop:sw}
The functor $F_N$ factors through the category $\Rep(\gl_N)$
of representations of the general linear Lie algebra $\gl_N$
\beql{compfun}
\xymatrix{ \cS \ar[rr]^{F_N} \ar[rd]_{SW_N} & & \Vect \\ &
\Rep(\gl_N) \ar[ru] }
\eeq
via a full monoidal functor $SW_N:\cS\to \Rep(\gl_N)$
and the forgetful functor
$\Rep(\gl_N)\to\Vect$.
\end{prop}

\bpf
Let $V$ be the $N$-dimensional vector representation of $\gl_N$.
By the classical Schur--Weyl duality
we have the homomorphisms
\beql{swfun}
k[S_n]\to \End_{\gl_N}(V^{\ot n})
\eeq
which extend to a monoidal
functor
$
SW_N:\cS\to \Rep(\gl_N)
$
sending the generator $X$ of $\cS$ to $V$, so that
$
SW_N(X^{\ot n})=V^{\ot n}.
$
Hence the functor $SW_N$ fits into the commutative diagram
\eqref{compfun}. The fullness of $SW_N$ follows from
the surjectivity of the homomorphisms \eqref{swfun}.
\epf

We will call $SW_N$ the {\it Schur--Weyl functor\/}.
In what follows we will consider its quantum and affine analogues
and establish similar properties.

\section{Hecke algebras and quantized enveloping algebras}\label{heck}
\setcounter{equation}{0}

\subsection{Braid groups and Hecke algebras}

Recall that the {\it braid group} $B_n$ is the group
generated by elements $t_1,\dots,t_{n-1}$ subject to the relations
$$
t_it_{i+1}t_i=t_{i+1}t_it_{i+1},\qquad t_it_j=t_jt_i\quad
\text{for}\quad |i-j|>1.
$$
The assignments
$$t_i\otimes 1\mapsto t_i,\quad 1\otimes t_j\mapsto t_{j+m},$$
define homomorphisms of group algebras
$$
k[B_m]\otimes k[B_n]\to k[B_{m+n}].
$$
These homomorphisms satisfy the associativity
axiom so that we get a multiplicative sequence
$k[B_*]=\{k[B_n]\ |\ n\geqslant 0\}$ as defined
in Sec.~\ref{subsec:ms}. The following proposition
is immediate from the presentations of the braid groups.

\begin{prop}\label{genb}
The multiplicative sequence $k[B_*]$ is generated
by the element $t\in k[B_2]$
subject to the relation
$$
\mu_{2,1}(t\otimes 1)\mu_{1,2}(1\otimes t)
\mu_{2,1}(t\otimes 1) = \mu_{1,2}(1\otimes t) \mu_{2,1}(t\otimes
1)\mu_{1,2}(1\otimes t)\qquad  \text{in}\quad  k[B_3].
$$
\end{prop}

Theorem~\ref{gr} and Proposition~\ref{genb}
imply a description of the monoidal category associated
with the multiplicative sequence $k[B_*]$.

\bth\label{thm:braidu}
The abelian monoidal category $\cB=\cC(k[B_*])$ is generated by an
object $X$ and an automorphism $c:X\otimes X\to X\otimes X$ such
that
\begin{equation}\label{yb}
(c\otimes 1_X)(1_X\otimes c)(c\otimes 1_X) = (1_X\otimes c)(c\otimes
1_X)(1_X\otimes c).
\end{equation}
\eth

An automorphism $c$ satisfying
(\ref{yb}) will be called a {\it Yang--Baxter operator}, and the
corresponding object
$X$ -- a {\it Yang--Baxter object}.
Theorem~\ref{thm:braidu}
states that $\cB$ is a free monoidal category generated by a
Yang--Baxter object.

Let $c_{m,n}\in B_{m+n}$ be the braid
whose geometric presentation in terms of strands
is illustrated below; the first $m$ strands pass on
top of the remaining $n$ strands:

$$\xygraph{ !{0;/r2.5pc/:;/u2.5pc/::}
*+{}="t1"
:@{-}[d(2)r(2)]="d1"
:@{}[r(.5)]*{\dots}
:@{}[r(.5)]="d2"
:@{-}[u(2)l(2)]="t2"
(
:@{}[l(.5)]*{\dots}
,
:@{}[r]
:@{-}[d(2)l(2)]  |!{"d2";"t2"}\hole  |!{"t1";"d1"}\hole
:@{}[r(.5)]*{\dots}
:@{}[r(.5)]
:@{-}[u(2)r(2)] |!{"t1";"d1"}\hole  |!{"d2";"t2"}\hole
:@{}[l(.5)]*{\dots}
)
}
$$

\bigskip
\noindent
The collection $\{c_{m,n}\ |\ m,n\geqslant 0\}$ of the
isomorphisms $c_{m,n}:X^{\otimes
m}\otimes X^{\otimes n}\to X^{\otimes n}\otimes X^{\otimes m}$
defines a braiding on the category $\cB$. The
following freeness property of the category $\cB$ is well-known;
see \cite{js:bt}.

\begin{prop}
The category $\cB$ is a free abelian braided monoidal
category generated by one object.
\end{prop}

Suppose that $q$ is a nonzero element of the field $k$.
The {\it Hecke algebra} $\He_n(q)$ is the quotient of the group
algebra $k[B_n]$ by the ideal generated by
the elements $(t_i-q)(t_i+q^{-1})$ with
$i=1,\dots,n-1$. Equivalently, $\He_n(q)$ is
the associative algebra generated by
elements $t_1,\dots,t_{n-1}$ subject to the relations
\ben
(t_i-q)(t_i+q^{-1})=0,\qquad t_it_{i+1}t_i
=t_{i+1}t_it_{i+1},\qquad t_it_j=t_jt_i
\quad\text{for}\quad |i-j|>1.
\een

By taking the quotients of the group algebras $k[B_n]$
and using the multiplicative structure on $k[B_*]$
we get a multiplicative sequence of algebras
$\He_*(q)=\{\He_n(q)\ |\ n\geqslant 0\}$. Using the
presentation of $\He_n(q)$ we come to the following.

\begin{prop}\label{genh}
The multiplicative sequence $\He_*(q)$ is generated by an
element $t\in\He_2(q)$ subject to the relations
\begin{alignat}{2}
(t_i-q)(t_i+q^{-1})&=0\qquad  &&\text{in}\quad \He_2(q),
\non\\
\mu_{2,1}(t\otimes 1)\mu_{1,2}(1\otimes t)
\mu_{2,1}(t\otimes 1) &= \mu_{1,2}(1\otimes t) \mu_{2,1}(t\otimes
1)\mu_{1,2}(1\otimes t)\qquad
&&\text{in}\quad \He_3(q).
\non
\end{alignat}
\end{prop}

We call the category $\caH(q) = \cC(\He_*(q))$
associated with the multiplicative sequence $\He_*(q)$
the {\it Hecke category}. The next theorem
implied by Theorem~\ref{gr} and Proposition~\ref{genh} provides
its description as a monoidal category.

\begin{thm}\label{thm:heckec}
The Hecke category is generated as a monoidal category by an
object $X$ and an automorphism $c:X\otimes X\to X\otimes X$
subject to the relations
\begin{equation}\label{hyb}
(c-q)(c+q^{-1})=0,\qquad (c\otimes 1_X)(1_X\otimes c)(c\otimes 1_X)
= (1_X\otimes c)(c\otimes 1_X)(1_X\otimes c).
\end{equation}
\end{thm}

An automorphism $c$ satisfying
(\ref{hyb}) will be called a {\it Hecke Yang--Baxter operator}, and
the corresponding object
$X$ -- a {\it Hecke
Yang--Baxter object}. Theorem~\ref{thm:heckec} states that $\caH(q)$ is
a free abelian monoidal category generated by
a Hecke Yang--Baxter object. The following proposition
is analogous to Proposition~\ref{prop:free-ab-sym}
and whose proof we also omit.

\begin{prop}
The category $\caH(q)$ is a free abelian braided monoidal category
generated by a Hecke Yang--Baxter object.
\end{prop}

\subsection{Fiber functors and quantized enveloping algebras}
\label{subsec:ffhe}

Let $V$ be an
$N$-dimensional vector space with the basis $e_1,\dots,e_N$.
Following \cite{j:qu}, define the linear operator
$R:V^{\otimes 2}\to V^{\otimes 2}$ by
\beql{dj}
R(e_i\otimes e_j) =
\begin{cases}
e_j\otimes e_i &\qquad\text{if}\quad i<j,\\
q\ts e_j\otimes e_i &\qquad\text{if}\quad i=j,\\
e_j\otimes e_i + (q-q^{-1})\ts e_i\otimes
e_j &\qquad\text{if}\quad i>j.
\end{cases}
\eeq
Label the copies of $V$ in the tensor product $V^{\ot n}$
by $1,\dots,n$, respectively.
Let $R_{\tss l,l+1}$ denote the operator in this
tensor product space which acts as $R$ in the tensor
product of the copies of $V$ labeled by $l$ and $l+1$
and acts as the identity operator in the remaining
copies of $V$. By the results of \cite{j:qu},
the mapping $t_l\mapsto R_{\tss l,l+1}$ defines
a representation of the Hecke algebra $\He_n(q)$
in $V^{\ot n}$ and the image of $\He_n(q)$ in the endomorphism
algebra of $V^{\ot n}$ commutes with the image of an action of
the quantized enveloping algebra $\U_q(\gl_N)$
associated with $\gl_N$. Thus we get
a monoidal functor
$
J_N:\caH(q)\to \U_q(\gl_N)\da\Mod
$
sending the generator $X$ of $\caH(q)$ to $V$, so that
$
J_N(X^{\ot n})=V^{\ot n}.
$
We call $J_N$ the {\it Jimbo functor\/}. It
is a quantum analogue of the Schur--Weyl functor $SW_N$; cf.
Sec.~\ref{subsec:ffcla}. The following is a respective analogue
of Proposition~\ref{prop:sw}, where $F_N:\caH(q)\to\Vect$
denotes the fiber functor labeled by $N$.

\begin{prop}
The functor $F_N$ factors through the category
$\U_q(\gl_N)\da\Mod$ of representations of the
quantized enveloping algebra
$$\xymatrix{ \caH(q) \ar[rr]^{F_N}
\ar[rd]_{J_N} & & \Vect \\ & \U_{q}(\gl_N)\da\Mod \ar[ru] }$$
via a monoidal functor
$J_N:\caH(q)\to \U_q(\gl_N)\da\Mod$
and the forgetful functor\newline
$\U_q(\gl_N)\da\Mod\to\Vect$.
\end{prop}

As follows from \cite{j:qu}, the Jimbo functor $J_N$ is full
under some additional conditions
($k=\mathbb{C}$ and $q$ is not a root of unity).
On the other hand, if $k$ and $q$ are arbitrary,
then an analogous property of $J_N$
can be established by replacing the quantized enveloping algebra
with an appropriate integral version;
see \cite{dps:qw}.

\section{Affine symmetric groups and loop algebras}
\label{affsym}
\setcounter{equation}{0}

Let $A$ be a unital associative algebra. Define the sequence of algebras
$SA_n=A^{\otimes n}*S_n$ which are the {\it cross-products\/} of
the symmetric group
algebras with the tensor powers of $A$ with respect to the natural
permutation actions of $S_n$ on $A^{\otimes n}$. As a
vector space, the cross-product $A^{\otimes n}*S_n$
coincides with the tensor product $A^{\otimes n}\otimes k[S_n]$.
However, the algebra structure of $A^{\otimes n}*S_n$ is
different to the tensor product algebra. We will emphasize this fact
by using the notation $(a_1\otimes\dots\otimes a_n)*\sigma$ for the
element of $A^{\otimes n}*S_n$ corresponding to
$(a_1\otimes\dots\otimes a_n)
\otimes\sigma\in A^{\otimes n}\otimes k[S_n]$.
The multiplication in $A^{\otimes n}*S_n$ is given by the
following rule:
\beql{multrule}
\big((a_1\otimes\dots\otimes a_n)*\sigma\big)
\big((b_1\otimes\dots\otimes b_n)*\tau\big)
=\big(a_1b_{\sigma^{-1}(1)}
\otimes\dots\otimes a_nb_{\sigma^{-1}(n)}\big)*\sigma\tau.
\eeq

The multiplicative structure on the sequence of symmetric group
algebras extends to a multiplicative structure on $SA_n$:
$$
\mu_{m,n}:SA_m\otimes SA_n\to SA_{m+n},\qquad m,n\geqslant 0
$$
with
\begin{multline}
\mu_{m,n}\big(((a_1\otimes\dots\otimes a_m)*\sigma)
\otimes((b_1\otimes\dots\otimes b_n)*\tau)\big)\\
{} = (a_1\otimes\dots\otimes a_m\otimes
b_1\otimes\dots\otimes b_n)*\mu_{m,n}(\sigma\otimes\tau).
\non
\end{multline}

\begin{prop}\label{gena}
The multiplicative sequence $SA_*$ is generated by
$t\in S_2\subset SA_2$ and the elements of $SA_1=A$
subject to the relations:
\begin{alignat}{2}
t^2&=1\qquad  &&\text{in}\quad SA_2,
\non\\
t\ts\mu_{1,1}(u\otimes 1)&= \mu_{1,1}(1\otimes u)\ts t\qquad
&&\text{in}\quad SA_2,
\non
\end{alignat}
for any $u\in
SA_1$, and
$$\mu_{2,1}(t\otimes 1)\mu_{1,2}(1\otimes t)
\mu_{2,1}(t\otimes 1) = \mu_{1,2}(1\otimes t) \mu_{2,1}(t\otimes
1)\mu_{1,2}(1\otimes t)\quad  \text{in}\quad  SA_3.
$$
\end{prop}
\begin{proof}
This follows from the presentation for $SA_n$
where all elements of $A^{\ot n}$ and the elements
$t_1,\dots,t_{n-1}$ are taken as generators.
The defining relations are obtained by
using the multiplication rule \eqref{multrule} and
the standard presentation of $S_n$.
\end{proof}

We let $\cS(A)$ denote the Schur--Weyl category
$\cC(SA_*)$ associated with the multiplicative
sequence $SA_*$.
The following theorem is immediate
from Theorem \ref{gr} and Proposition \ref{gena}.

\begin{thm}\label{fsca}
The  category $\cS(A)$ is a monoidal category
generated by an object $X$
such that $\End_{\cS(A)}(X) = A$ and an automorphism $c:X\otimes
X\to X\otimes X$ subject to the relations
\begin{equation}
c^2=1,\qquad (c\otimes 1_X)(1_X\otimes c)(c\otimes 1_X) = (1_X\otimes
c)(c\otimes 1_X)(1_X\otimes c)
\end{equation}
and
$$
(a\otimes b)\ts c = c\ts (b\otimes a),\qquad a,b\in A.
$$
In other words, the category $\cS(A)$ is a free symmetric category
generated by an object $X$ such that $\End_{\cS(A)}(X) = A$.
\qed
\end{thm}

\begin{example}
{\it Monoidal autoequivalences of $\cS(A)$}.
Let $f$ be an automorphism of the algebra $A$.
By Theorem~\ref{fsca},
the assignment
$(F(X),F(a),F(c)):= (X,f(a),c)$ for all $a\in A$
defines a monoidal functor
$T_f:\cS(A)\to\cS(A)$. The composition
$T_f\circ T_g$ is canonically
isomorphic (as a monoidal functor) to $T_{fg}$. Clearly, $T_1$ is
canonically isomorphic to the identity functor. Thus, $T_f$
is a monoidal autoequivalence of $\cS(A)$
for any automorphism $f$ of $A$.

Note that in the particular case when $A=k[u]$, the automorphism group
$\Aut(k[u])$ is isomorphic to the semi-direct product $k^*\ltimes k$:
any automorphism of $k[u]$ has the form $\phi_{a,b}$ with
$a\in k^*$ and $b\in k$, and $\phi_{a,b}(u)=au+b$.
Later on we will consider quantum deformations
of the category $\cS(k[u])$; namely, the categories corresponding
to multiplicative sequences of the affine Hecke algebras
and their degenerate versions. In those cases, the automorphism
group $\Aut(k[u])$ will be reduced to one-parameter subgroups
with $b=0$ and $a=1$, respectively; see also
Examples~\ref{ex:mel} and \ref{ex:mehq}.
\qed
\end{example}

\subsection{Affine symmetric groups and
free affine symmetric category}
\label{subsec:asc}

Here we consider two particular cases of the general
construction described above. We will take
$A$ to be the algebra of
Laurent polynomials $k[u^{\pm 1}]=k[u,u^{-1}]$ in a variable $u$ and
the algebra of polynomials $k[u]$.

Let $AS_n$ be the cross-product algebra
$k[u_1^{\pm 1},\dots,u_n^{\pm 1}]*S_n$
with respect to the natural permutation
action of $S_n$ on $k[u_1^{\pm 1},\dots,u_n^{\pm 1}]$. In other words,
$AS_n$ is generated by the elements $t_1,\dots,t_{n-1}$
and invertible elements $u_1,\dots,u_n$
subject to the relations \eqref{cox} together with
\ben
t_i\tss u_i\tss t_i=u_{i+1},\qquad u_i\tss u_j=u_j\tss u_i.
\een
Similarly, we
let $SAS_n$ denote the cross-product algebra $k[u_1,\dots,u_n]*S_n$
with respect to the natural permutation action of $S_n$
on $k[u_1,\dots,u_n]$. The algebra $SAS_n$ can be presented
in the same way
as $AS_n$, with the invertibility condition
on the $u_i$ omitted.

Now define the {\it affine symmetric category\/}
$\AS$ to be the monoidal
category $\cC(AS_*)$ corresponding to the multiplicative sequence
$AS_*=\{AS_n\ |\ n\geqslant 0\}$.
Similarly, define the {\it semi-affine symmetric
category\/}
$\SAS$ to be the monoidal category $\cC(SAS_*)$ corresponding to the
multiplicative sequence $SAS_*=\{SAS_n\ |\ n\geqslant 0\}$.

Using the presentations of the algebras $AS_n$ and $SAS_n$
we obtain the following.

\begin{prop}\label{genas}
The multiplicative sequence $AS_*$ is generated by $t\in AS_2$
and an invertible element $u\in AS_1$ subject to the relations
\begin{alignat}{2}
t^2&=1\qquad  &&\text{in}\quad AS_2,
\non\\
t\tss\mu_{1,1}(u\otimes 1)&= \mu_{1,1}(1\otimes u)\tss t\qquad
&&\text{in}\quad AS_2,
\non\\
\mu_{2,1}(t\otimes 1)\mu_{1,2}(1\otimes t)
\mu_{2,1}(t\otimes 1) &= \mu_{1,2}(1\otimes t) \mu_{2,1}(t\otimes
1)\mu_{1,2}(1\otimes t)\qquad  &&\text{in}\quad  AS_3.
\non
\end{alignat}
Moreover, the multiplicative sequence $SAS_*$ admits
the same presentation with the invertibility condition
on $u\in SAS_1$ omitted.
\end{prop}

The next corollary is immediate from Theorem~\ref{fsca}.

\begin{cor}
The affine symmetric category $\AS$ is a free
symmetric monoidal category generated
by an object $X$ and an automorphism $x:X\to X$. Moreover, the
semi-affine symmetric category $\SAS$ is
a free symmetric monoidal category
generated by an object $X$ and an endomorphism $x:X\to X$.
\qed
\end{cor}

\subsection{Fiber functors from $\cS(A)$}

We now return to the Schur--Weyl category $\cS(A)$
associated with an arbitrary unital associative algebra $A$
as defined in the beginning of Sec.~\ref{affsym}.

The natural embeddings $k[S_n]\hra SA_n$ define
a homomorphism of multiplicative sequences $k[S_*]\to
SA_*$. This gives rise to a monoidal functor $\cS\to\cS(A)$
which sends the generator $X\in\cS$ to the generator $X\in\cS(A)$ and
sends the Yang--Baxter operator
$c\in \End_\cS(X^{\otimes 2})$ to $c\in
\End_{\cS(A)}(X^{\otimes 2})$.
The algebras $A^{\otimes n}*S_n$
admit natural multiplicative decompositions
with the components $A^{\otimes n}$ and $k[S_n]$.
Hence, applying Theorem~\ref{aff} we derive that the right adjoint
$F:\cS(A)\to\cS$ is also
monoidal. By Theorem~\ref{fsca}, the functor $F$
is determined by its values on
the generating object $X\in\cS(A)$ together with its values on the
generating morphisms $a\in A= \End_{\cS(A)}(X)$ and $c\in
\End_{\cS(A)}(X^{\otimes 2})$.
The object $F(X)$ is the tensor product $X\otimes A$ of
the generator $X\in\cS$ with the underlying vector space of the
algebra $A$. For any element $a\in A$ the
endomorphism $F(a)$ is the morphism induced by the linear map $A\to
A$, which is the left multiplication by $a$. The automorphism
$F(c)$ is $c\otimes t$, where we identify
$\End_\cS((X\otimes A)^{\otimes 2})$ with the tensor product
$\End_\cS(X^{\otimes 2})\otimes \End(A^{\otimes 2})$ and
$t:A^{\otimes 2}\to A^{\otimes 2}$ is the transposition
$a\otimes b\mapsto b\otimes a$.

Composing the monoidal functor $\cS(A)\to\cS$ with a fiber functor
$\cS\to\Vect$ we get a fiber functor $\cS(A)\to\Vect$. In
particular, taking the symmetric fiber functor $F_N\se\cS\to\Vect$
we get a fiber functor $F_N\se\cS(A)\to\Vect$.
Applying Theorem~\ref{fsca} as above, we find that
$F_N$ is determined by its values $F_N(X)$, $F_N(c)$ and $F_N(a)$
for all $a\in A$. Introducing the vector space
$V=k^N$ we find that
the object $F_N(X)$ is the tensor product $V\otimes A$.
For any $a\in A$ the endomorphism
$F_N(a)$ is the morphism, induced by the linear map $A\to A$ which
is the left multiplication by $a$. The automorphism $F_N(c)$ is the
transposition
$$
t:(V\ot A)\ot(V\ot A)\to (V\ot A)\ot(V\ot A),\quad
(v\otimes a)\otimes(u\otimes b)\mapsto
(u\otimes b)\otimes(v\otimes a).
$$

The following is an analogue of Proposition~\ref{prop:sw}
providing an affine version of the Schur--Weyl functor $SW_N$
and it is verified in the same way.

\begin{prop}
The monoidal functor $F_N:\cS(A)\to\Vect$ factors through the
category $\Rep(\gl_N(A))$ of representations of the general
linear Lie algebra over $A$
$$\xymatrix{ \cS(A) \ar[rr]^{F_N} \ar[rd]_{SW_N} & &
\Vect \\ &  \Rep(\gl_N(A)) \ar[ru] }$$ via a monoidal
functor $SW_N:\cS(A)\to \Rep(\gl_N(A))$
and the forgetful functor\newline
$\Rep(\gl_N(A))\to\Vect$.
\end{prop}

\bre
We believe that the functor $SW_N$ is full. Although we do not
have a proof of this conjecture.
\ere

\section{Yangians and degenerate affine Hecke algebras}
\label{daffhe}
\setcounter{equation}{0}

\subsection{Degenerate affine Hecke algebras}

The {\it degenerate affine Hecke algebra} $\Lambda_n$ is the
unital associative algebra generated by elements
$t_1,\dots,t_{n-1}$ and $y_1,\dots,y_n$
subject to the relations
\begin{alignat}{2}
t_i^2=1,\qquad t_it_{i+1}t_i&=t_{i+1}t_it_{i+1},
\qquad &&t_it_j=t_jt_i\quad\text{for}\quad |i-j|>1,
\non\\
y_it_i-t_iy_{i+1}&=1,\qquad &&y_iy_j=y_jy_i.
\non
\end{alignat}
The assignments
\ben
\bal
t_i&\otimes 1\mapsto t_i,\qquad 1\otimes t_j\mapsto t_{j+m},\\
y_i&\otimes 1\mapsto y_i,\qquad 1\otimes y_j\mapsto y_{j+m}
\eal
\een
define algebra homomorphisms
\ben
\Lambda_m\otimes\Lambda_n\to\Lambda_{m+n}.
\een
It is easy to see that
these homomorphisms satisfy the associativity axiom
thus giving rise to the multiplicative sequence of algebras
$\Lambda_*=\{\Lambda_n\ |\ n\geqslant 0\}$; see Sec.~\ref{subsec:ms}.
Hence we get
a monoidal category $\caL=\cC(\Lambda_*)$
which we call the {\it degenerate
affine Hecke category}.

\begin{prop}\label{gendah}
The multiplicative sequence $\Lambda_*$ is generated by
elements
$y\in\Lambda_1$ and $t\in\Lambda_2$ subject to the relations
\begin{alignat}{2}
t^2&=1\qquad  &&\text{in}\quad \Lambda_2,
\non\\
t\tss\mu_{1,1}(y\otimes 1) - \mu_{1,1}(1\otimes y)\tss t
&= 1\qquad &&\text{in}\quad \Lambda_2,
\non\\
\mu_{2,1}(t\otimes 1)\mu_{1,2}(1\otimes t)
\mu_{2,1}(t\otimes 1) &= \mu_{1,2}(1\otimes t) \mu_{2,1}(t\otimes
1)\mu_{1,2}(1\otimes t) \qquad  &&\text{in}\quad \Lambda_3.
\non
\end{alignat}
\end{prop}

\begin{proof}
This follows from the presentation for $\Lambda_n$.
\end{proof}

\begin{thm}\label{fpdahc}
The degenerate affine Hecke category $\caL$ is a free monoidal
category generated by one object $X$,
an endomorphism $x:X\to X$ and an
involutive Yang--Baxter operator $c:X^{\otimes 2}\to X^{\otimes 2}$
subject to the relations
\ben
c^2=1,\qquad (c\otimes 1_X)(1_X\otimes c)(c\otimes 1_X)
= (1_X\otimes c)(c\otimes 1_X)(1_X\otimes c),
\een
and
\beql{dege}
(x\otimes 1)\tss c - c\tss (1\otimes x) = 1.
\eeq
\end{thm}

\begin{proof}
This is immediate from Theorem~\ref{gr} and
Proposition~\ref{gendah}.
\end{proof}

Theorem~\ref{fpdahc} implies that
monoidal functors from $\caL$ to a monoidal
category $\cC$ correspond to triples $(V,x,c)$,
where $V$ is an object in $\cC$, $x\in \End_\cC(V)$ is its
endomorphism, and $c\in \End(V^{\otimes 2})$ is an involutive
Yang--Baxter operator satisfying \eqref{dege}.

\begin{example}\label{ex:mel}
{\it Monoidal autoequivalences of $\caL$}.
Let $u$ be an element of the basic field $k$. The triple
$(X,x+uI,c)$ defines a monoidal functor
$T_u:\caL\to\caL$. The composition
$T_u\circ T_v$ is canonically isomorphic (as a monoidal functor) to
$T_{u+v}$. Hence $T_u$ is a monoidal autoequivalence of $\caL$.
\qed
\end{example}

Due to \cite{l:aha}, the degenerate affine Hecke algebra $\Lambda_n$
admits multiplicative decompositions
\beql{dahamd}
\Lambda_n = k[y_1,\dots,y_n]\ts k[S_n] = k[S_n]\ts k[y_1,\dots,y_n].
\eeq
In the particular case $n=2$ the decompositions \eqref{dahamd}
follow from the relation
\ben
t\tss f(y_1,y_2) = f(y_2,y_1)\tss t
+ \frac{f(y_1,y_2)-f(y_2,y_1)}{y_1-y_2},
\een
where $f(y_1,y_2)$ is a polynomial in $y_1,y_2$.

The natural embeddings $k[S_n]\hra\Lambda_n$ define
a homomorphism of multiplicative sequences $k[S_*]\to\Lambda_*$.
Hence, we get a monoidal functor $\cS\to\caL$
which sends the generator
$X\in\cS$ to the generator $X\in\caL$ and sends the Yang--Baxter
operator
$c\in \End_\cS(X^{\otimes 2})$ to $c\in \End_\caL(X^{\otimes
2})$. By Theorem \ref{aff}, its right adjoint $F:\caL\to\cS$ is also
monoidal. Furthermore, due to Theorem \ref{fpdahc}, the functor $F$
is determined by its values on
the generating object $X\in\caL$ together with its values on the
generating morphisms $x\in \End_\caL(X)$ and
$c\in \End_\caL(X^{\otimes
2})$. We now describe these values. The
object $F(X)$ is the tensor product $X\otimes k[y]$ of the generator
$X\in\cS$ with the vector space
of polynomials $k[y]$.
The endomorphism $F(x)$ is the morphism induced by the
linear map $k[y]\to k[y]$ which is the multiplication by $y$. To
describe the automorphism $F(c)$, identify
$\End_\cS((X\otimes k[y])^{\otimes 2})$ with the tensor product
\ben
\End_\cS(X^{\otimes 2})\ot\End(k[y_1,y_2])
\simeq k[S_2]\ot\End(k[y_1,y_2]).
\een
Here $\End(k[y_1,y_2])$ is
the algebra of $k$-endomorphisms of the vector space
$k[y_1,y_2]\simeq k[y]^{\otimes 2}$. Then we have
$F(c)=\partial+t\tss \tau$,
where $t\in S_2$ is the involution, $\partial\in\End(k[y_1,y_2])$
is the divided difference operator
\beql{divdiff}
\partial\tss f(y_1,y_2)  = \frac{f(y_1,y_2)-f(y_2,y_1)}{y_1-y_2}
\eeq
and $\tau$ is the algebra automorphism
of $k[y_1,y_2]$ defined on the generators by
$\tau(y_1) = y_2$ and $\tau(y_2) = y_1$.

\subsection{Fiber functors and Yangians}

Composing the monoidal functor $\caL\to\cS$ with a fiber functor
$\cS\to\Vect$ we get a fiber functor $\caL\to\Vect$. In particular,
taking the symmetric fiber functor $F_N:\cS\to\Vect$ we get a
fiber functor $F_N:\caL\to\Vect$ (for which we keep the same notation).
By Theorem \ref{fpdahc}, $F_N$ is
determined by its values on the generating object $X\in\caL$ together
with its values on the generating morphisms $x\in \End_\caL(X)
$ and $c\in
\End_\caL(X^{\otimes 2})$. Now we describe the triple
$(F_N(X),F_N(x),F_N(c))$. The object $F_N(X)$ is the
tensor product $V\otimes k[y]$, where $V=k^N$.
The endomorphism $F_N(x)$ is induced by the operator
of multiplication by $y$ in $k[y]$ so that
\beql{x}
F_N(x)(v\ot y^s) = v\ot  y^{s+1},\qquad v\in V.
\eeq
To describe $F_N(c)$, we identify
$F_N(X)\otimes F_N(X)$ with $V\otimes V\otimes
k[y_1,y_2]$ by
\ben
(v\ot y^r)\otimes (w\ot y^s)\mapsto v\otimes w\otimes
y_1^ry_2^s.
\een
Then for any $f\in k[y_1,y_2]$,
\beq\label{cact2}
F_N(c)(v\otimes
w\otimes f(y_1,y_2)) =  w\otimes v\otimes f(y_2,y_1) +v\otimes
w\otimes \partial f(y_1,y_2).
\eeq

In particular, it follows from the definition that the homomorphisms
$F_N(x)$ and $F_N(c)$ satisfy the relations
\ben
\bal
(F_N(x)\otimes 1)F_N(c) &- F_N(c)(1\otimes F_N(x)) = 1,\\
(F_N(c)\otimes 1)(1\otimes F_N(c))(F_N(c)\otimes 1)
&= (1\otimes F_N(c))(F_N(c)\otimes 1)(1\otimes F_N(c)).
\eal
\een

Now we recall some basic facts about the {\it Yangian\/}; see e.g.
\cite[Ch.~12]{cp:gq} and
\cite[Ch.~1]{m:yc} for more details.
The Yangian $\Y(\gl_N)$ is the unital
associative algebra generated by elements
$t^{(r)}_{ij}$ with $1\leqslant i,j\leqslant N$ and
$r = 1,2,\dots$ subject to the defining relations
\begin{equation}\label{yr}
[t^{(r+1)}_{ij},t^{(s)}_{kl}] - [t^{(r)}_{ij},t^{(s+1)}_{kl}] =
t^{(r)}_{kj}t^{(s)}_{il} - t^{(s)}_{kj}t^{(r)}_{il},
\end{equation}
where $r,s\geqslant 0$ and $t^{(0)}_{ij}=\delta_{ij}$.
The Yangian is a Hopf algebra with the coproduct
defined by
\begin{equation}\label{yc}
\Delta(t^{(r)}_{ij}) = \sum_{k=1}^N \sum_{s=0}^r
t^{(s)}_{ik}\otimes t^{(r-s)}_{kj}.
\end{equation}

The Hopf algebra $\Y(\gl_N)$ is a deformation of the
universal enveloping algebra $\U(\gl_N[y])$
in the class of Hopf algebras. As before, we let
$\{e_1,\dots,e_N\}$ denote a basis of an
$N$-dimensional vector space $V$. Then
the vector representation of $\gl_N$ in $V$
extends to a representation of $\gl_N[y]$ on
the vector space $V\otimes k[y]$,
and it gives rise to the representation of $\Y(\gl_N)$
on this space defined by
\beql{yangact}
t^{(r)}_{ij}(e_k\ot y^s) = \delta_{jk}e_i\ot y^{r+s-1}.
\eeq

\begin{prop}\label{prop:yangf}
The fiber functor $F_N:\caL\to\Vect$ factors through the category
of representations of the Yangian $\Y(\gl_N)\da\Mod$,
$$\xymatrix{ \caL \ar[rr]^{F_N} \ar[rd]_{D_N} & &
\Vect \\ & \Y(\gl_N)\da\Mod \ar[ru] }$$ via a monoidal
functor $D_N:\caL\to \Y(\gl_N)\da\Mod$
and the forgetful functor $\Y(\gl_N)\da\Mod\to\Vect$.
\end{prop}

\begin{proof}
We need to show that the maps $F_N(x)$ and $F_N(c)$
defined in (\ref{x}) and (\ref{cact2}) are morphisms of
$\Y(\gl_N)$-modules. This is obviously true
for $F_N(x)$. Furthermore,
writing the action \eqref{yangact} in terms of the formal
series
\ben
t_{ij}(u)=\delta_{ij}+\sum_{r=1}^{\infty}
t_{ij}^{(r)}u^{-r}
\een
we get
\ben
t_{ij}(u)(v\ot y^s)
=\Big(\delta_{ij}+\frac{e_{ij}}{u-y}\Big)(v\ot y^s),
\een
where the $e_{ij}\in\End(V)$ denote the standard matrix units.
By the
coproduct formula \eqref{yc}, we have
\ben
t_{ij}(u)\big(v\otimes w\otimes
f(y_1,y_2)\big)=\sum_{k=1}^N
\Big(\delta_{ik}+\frac{e_{ik}}{u-y_1}\Big)v\otimes
\Big(\delta_{kj}+\frac{e_{kj}}{u-y_2}\Big)w\otimes f(y_1,y_2).
\een
Hence, applying \eqref{cact2}, we get
\ben
\bal F_N(c)\tss t_{ij}(u)&\tss\big(v\otimes
w\otimes f(y_1,y_2)\big) =\sum_{k=1}^N
\Big(\delta_{kj}+\frac{e_{kj}}{u-y_1}\Big)w\otimes
\Big(\delta_{ik}+\frac{e_{ik}}{u-y_2}\Big)v\otimes f(y_2,y_1)\\
&+\delta_{ij}\big(v\otimes w\otimes \partial f(y_1,y_2)\big)
+\big(e_{ij}v\otimes w\otimes \partial\ts\frac{f(y_1,y_2)}{u-y_1}\big)\\
{}&+\big(v\otimes e_{ij}w\otimes
\partial\ts\frac{f(y_1,y_2)}{u-y_2}\big)
+\sum_{k=1}^N\Big(e_{ik}v\otimes e_{kj}w\otimes
\partial\ts\frac{f(y_1,y_2)}{(u-y_1)(u-y_2)}\big).
\eal
\een
On the
other hand,
\ben
\bal
t_{ij}(u)\ts F_N(c)\tss\big(v\otimes w\otimes f(y_1,y_2)\big)
&=\sum_{k=1}^N \Big(\delta_{ik}+\frac{e_{ik}}{u-y_1}\Big)w\otimes
\Big(\delta_{kj}+\frac{e_{kj}}{u-y_2}\Big)v\otimes f(y_2,y_1)\\
&+\sum_{k=1}^N\Big(\delta_{ik}+\frac{e_{ik}}{u-y_1}\Big)v\otimes
\Big(\delta_{kj}+\frac{e_{kj}}{u-y_2}\Big)w\otimes \partial
f(y_1,y_2).
\eal
\een
In order to compare these two expressions,
note that
\ben
\bal
\partial\ts\frac{f(y_1,y_2)}{u-y_1}
&=\frac{f(y_2,y_1)}{(u-y_1)(u-y_2)}+\frac{1}{u-y_1}\ts\partial
f(y_1,y_2),\\
\partial\ts\frac{f(y_1,y_2)}{u-y_2}
&=-\frac{f(y_2,y_1)}{(u-y_1)(u-y_2)}+\frac{1}{u-y_2}\ts\partial
f(y_1,y_2),
\eal
\een
and
\ben
\partial\ts\frac{f(y_1,y_2)}{(u-y_1)(u-y_2)}=
\frac{1}{(u-y_1)(u-y_2)}\ts\partial f(y_1,y_2).
\een
Therefore,
\ben
\bal
F_N(c)\tss& t_{ij}(u)\big(v\otimes w\otimes f(y_1,y_2)\big)
-t_{ij}(u)\ts F_N(c)\tss\big(v\otimes w\otimes f(y_1,y_2)\big)\\
{}&=\Big(e_{ij}v\otimes w-v\otimes e_{ij}w+\sum_{k=1}^N
\big(e_{kj}w\otimes e_{ik}v-e_{ik}w\otimes
 e_{kj}v\big)\Big)\otimes \frac{f(y_2,y_1)}{(u-y_1)(u-y_2)}.
\eal
\een
Taking the basis vectors $v=e_a$ and $w=e_b$ we find that
the expression in the brackets equals
\ben
\delta_{ja}e_i\otimes
e_b-\delta_{jb}e_a\otimes e_i+\delta_{jb}e_a\otimes e_i
-\delta_{ja}e_i\otimes e_b=0,
\een
thus completing the proof.
\end{proof}

The functor $D_N$ is called the {\it Drinfeld
functor}. The following property of $D_N$ was
essentially established in \cite{a:df}, \cite{d:da}.

\begin{prop}
The Drinfeld functor $D_N:\caL\to \Y(\gl_N)\da\Mod$ is full.
\end{prop}

\bpf
By the construction, $D_N$ sends the generator $X$ of $\caL$
to the vector representation $V[y]=V\ot k[y]$ of $\Y(\gl_N)$.
As $D_N$ is a monoidal functor,
we have $D_N(X^{\ot n})=(V[y])^{\ot n}$, while the effect
of $D_N$ on morphisms amounts to the collection
of homomorphisms
\ben
\Lambda_n\to \End_{\Y(\gl_N)}\big(V[y]\ot V[y]\ot\dots\ot V[y]\big).
\een
Note that in this description we actually work
in the category $\overline\cC(\Lambda_*)$ rather than
$\caL=\cC(\Lambda_*)$. Extending now $D_N$ to $\caL$
we come to the formula
\ben
D_N(M)=M\ot_{\Lambda_n}(V[y])^{\ot n},\qquad M\in \Mod\da\Lambda_n.
\een
The multiplicative decomposition $\Lambda_n=k[S_n]\ts k[y]^{\ot n}$
allows us to identify
\ben
M\ot_{\Lambda_n}(V[y])^{\ot n}\simeq M\ot_{k[S_n]}V^{\ot n}.
\een
Due to the results of \cite{a:df}, \cite{d:da}, the functors
\ben
\Mod\da\Lambda_n \longrightarrow \Y(\gl_N)\da\Mod,\qquad
M\mapsto M\ot_{k[S_n]}V^{\ot n}
\een
are full.
\epf

\section{Quantum affine algebras and affine Hecke algebras}
\label{affhe}
\setcounter{equation}{0}

\subsection{Affine braid groups and affine Hecke algebras}

The {\it affine braid group} $\wt B_n$ is the group with
generators $t_1,\dots,t_{n-1}$ and $y_1,\dots,y_n$ subject
to the defining relations
\begin{alignat}{2}
t_it_{i+1}t_i&=t_{i+1}t_it_{i+1},
\qquad &&t_it_j=t_jt_i\quad\text{for}\quad |i-j|>1,
\non\\
t_iy_it_i&=y_{i+1},\qquad &&y_iy_j=y_jy_i.
\non
\end{alignat}

The assignments
\ben
\bal
t_i&\otimes 1\mapsto t_i,\qquad 1\otimes t_j\mapsto t_{j+m},\\
y_i&\otimes 1\mapsto y_i,\qquad 1\otimes y_j\mapsto y_{j+m}
\eal
\een
define algebra homomorphisms
$$
k[\wt B_m]\otimes k[\wt B_n]\to k[\wt B_{m+n}]
$$
These homomorphisms satisfy the associativity axiom and so give rise
to the multiplicative sequence of algebras
$k[\wt B_*]=\{k[\wt B_n]\ |\ n\geqslant 0\}$; see Sec.~\ref{subsec:ms}.
Hence we get
a monoidal category $\wt\cB = \cC(k[\wt B_*])$
which we call the {\it affine braid category}.

\begin{prop}\label{genab}
The multiplicative sequence $k[\wt B_n]$ is generated by
elements $y\in
k[\wt B_1]$ and $t\in k[\wt B_2]$ subject to the relations
\begin{alignat}{2}
t\tss\mu_{1,1}(y\otimes 1)\tss t &= \mu_{1,1}(1\otimes y)
\qquad &&\text{in}\quad k[\wt B_2],
\non\\
\mu_{2,1}(t\otimes 1)\mu_{1,2}(1\otimes t)\mu_{2,1}(t\otimes 1)
&= \mu_{1,2}(1\otimes t)\mu_{2,1}(t\otimes 1)\mu_{1,2}(1\otimes t)
\qquad  &&\text{in}\quad k[\wt B_3].
\non
\end{alignat}
\end{prop}

\begin{proof}
This follows from the presentation for $\wt B_n$.
\end{proof}

\begin{thm}
The affine braid category is a free monoidal
category generated by one object $X$,
an endomorphism $x:X\to X$ and a Yang--Baxter
operator $c:X^{\otimes 2}\to X^{\otimes 2}$
subject to the relations
\ben
(c\otimes 1_X)(1_X\otimes c)(c\otimes 1_X)
= (1_X\otimes c)(c\otimes 1_X)(1_X\otimes c)
\een
and
\ben
c\tss (x\otimes 1)\tss c =1\otimes x.
\een
\end{thm}

\begin{proof}
This follows from Theorem~\ref{gr} and Proposition~\ref{genab}.
\end{proof}

Now we define certain quotients of the affine braid group algebras.
Fix a nonzero element $q\in k$.
The {\it affine Hecke algebra} $\wt \He_n(q)$ is the associative algebra
generated by elements $t_1,\dots,t_{n-1}$ and
invertible elements $y_1,\dots,y_n$ subject to the relations
\begin{alignat}{2}
(t_i-q)(t_i+q^{-1})=0,\qquad
t_it_{i+1}t_i&=t_{i+1}t_it_{i+1},\qquad &&t_it_j=t_jt_i\quad
\text{for}\quad |i-j|>1,
\non\\
t_i\tss y_i\tss t_i&=y_{i+1},\qquad &&y_i\tss y_j=y_j\tss y_i.
\non
\end{alignat}
The assignments
\ben
\bal
t_i&\otimes 1\mapsto t_i,\qquad 1\otimes t_j\mapsto t_{j+m},\\
y_i&\otimes 1\mapsto y_i,\qquad 1\otimes y_j\mapsto y_{j+m}
\eal
\een
define algebra homomorphisms
$$
\mu_{m,n}:\wt \He_m(q)\otimes\wt \He_n(q)\to\wt \He_{m+n}(q)
$$
making the sequence
$\wt \He_*(q)=\{\wt \He_n(q)\ |\ n\geqslant 0\}$
into a multiplicative sequence of algebras;
see Sec.~\ref{subsec:ms}. This
gives rise to a monoidal category $\wt\caH(q)=\cC(\wt \He_*(q))$,
which we call the {\it affine Hecke category}.

\begin{prop}\label{genah}
The multiplicative sequence $\wt \He_*(q)$ is generated by
elements $y\in\wt \He_1(q)$ and $t\in\wt \He_2(q)$ subject to the
relations
\begin{alignat}{2}
(t-q)(t+q^{-1})&=0\qquad  &&\text{in}\quad \wt\He_2(q),
\non\\
t\tss\mu_{1,1}(y\otimes 1)\tss t &
= \mu_{1,1}(1\otimes y)\qquad &&\text{in}\quad \wt\He_2(q),
\non\\
\mu_{2,1}(t\otimes 1)\mu_{1,2}(1\otimes t)\mu_{2,1}(t\otimes 1)
&= \mu_{1,2}(1\otimes t)\mu_{2,1}(t\otimes 1)\mu_{1,2}(1\otimes t)
\qquad  &&\text{in}\quad \wt\He_3(q).
\non
\end{alignat}
\end{prop}

\begin{proof}
This follows from the presentation of $\wt \He_n(q)$.
\end{proof}

\begin{thm}
The affine Hecke category is a free monoidal
category generated by one object
$X$, an endomorphism $x:X\to X$ and a Hecke Yang--Baxter operator
$c:X^{\otimes 2}\to X^{\otimes 2}$
subject to the relations
\ben
(c-q)(c+q^{-1})=0,\qquad (c\otimes 1_X)(1_X\otimes c)(c\otimes 1_X)
= (1_X\otimes c)(c\otimes 1_X)(1_X\otimes c)
\een
and
\ben
c\tss(x\otimes 1)\tss c =1\otimes x.
\een
\end{thm}
\begin{proof}
This follows from Theorem~\ref{gr} and Proposition~\ref{genah}.
\end{proof}

\begin{cor}\label{fahc}
Monoidal functors from the affine braid category
{\rm(}resp., from the affine Hecke category{\rm)} to a
monoidal category $\cC$ are determined by triples
$(V,x,c)$, where $V$ is
an object in $\cC$, $x\in \End_\cC(V)$ is its endomorphism, and $c\in
\End(V^{\otimes 2})$ is a {\rm(}Hecke{\rm)} Yang--Baxter
operator such that
\begin{equation}\label{ah}
c\tss(x\otimes 1)\tss c = 1\otimes x.
\end{equation}
\end{cor}

\begin{example}\label{ex:mehq}
{\it Monoidal autoequivalences of $\wt\caH(q)$}.
Let $u$ be an invertible element of the basic field $k$. The triple
$(V,u\tss x,c)$ satisfies the conditions of Corollary \ref{fahc} and so
it defines a monoidal functor $T_u:\wt\caH(q)\to\wt\caH(q)$. The
composition $T_u\circ T_v$ is canonically isomorphic (as a monoidal
functor) to $T_{uv}$. Hence $T_u$ is a monoidal autoequivalence of
$\wt\caH(q)$.
\qed
\end{example}

Due to \cite[Lemma~3.4]{l:aha}, affine Hecke algebras
admit multiplicative decompositions
\begin{equation}\label{ahamd}
\wt \He_n(q) = k[y_1^{\pm 1},\dots,y_n^{\pm 1}]\ts\He_n(q)
= \He_n(q)\ts k[y_1^{\pm 1},\dots,y_n^{\pm 1}].
\end{equation}
For $n=2$ the decomposition (\ref{ahamd}) follows from the
relation
\ben
t\tss f(y_1,y_2) = f(y_2,y_1)\tss t - (q-q^{-1})\ts y_2\ts
\frac{f(y_1,y_2)-f(y_2,y_1)}{y_1-y_2},
\een
where $f(y_1,y_2)$ is an arbitrary Laurent polynomial in $y_1,y_2$.

The monoidal functor $\caH(q)\to\wt\caH(q)$ defined by
the natural homomorphism of multiplicative sequences $\He_*(q)\to\wt
\He_*(q)$,
sends the generator $X\in\caH(q)$ to the generator
$X\in\wt\caH(q)$ and sends $c\in
\End_{\caH(q)}(X^{\otimes 2})$ to $c\in
\End_{\wt\caH(q)}(X^{\otimes 2})$. By Theorem \ref{aff}, its right
adjoint $F:\wt\caH(q)\to\caH(q)$ is also monoidal.
By Corollary~\ref{fahc},
the functor $F$ is determined by its values on the generating object
$X\in\wt\caH(q)$ together with its values on the generating
morphisms $x\in \End_{\wt\caH(q)}(X)$ and $c\in
\End_{\wt\caH(q)}(X^{\otimes 2})$. Now we describe these values.
The object $F(X)$ is the tensor product
$X\ot k[y^{\pm 1}]$ of the generator $X\in\caH(q)$ with the
vector space $k[y^{\pm}]$ of Laurent polynomials. The
automorphism $F(x)$ is the morphism induced by the linear map
$k[y^{\pm 1}]\to k[y^{\pm 1}]$, which is multiplication by $y$. To
describe the automorphism $F(c)$, identify
$\End_{\wt\caH(q)}((X\ot k[y^{\pm 1}])^{\otimes 2})$ with the
tensor product
\ben
\End_{\caH(q)}(X^{\otimes 2})\ot\End(k[y_1^{\pm},y_2^{\pm}])
\simeq \He_2(q)\ot\End(k[y_1^{\pm},y_2^{\pm}]).
\een
Here
$\End(k[y_1^{\pm},y_2^{\pm}])$ is the algebra of $k$-endomorphisms
of the vector space of Laurent polynomials
$k[y_1^{\pm},y_2^{\pm}]\simeq
k[y^{\pm}]^{\otimes 2}$. We have $F(c)=d+t\ts\tau$, where $t=t_1$
is the generator of
$\He_2(q)$, $\tau\in \End(k[y_1^{\pm},y_2^{\pm}])$ is
the algebra automorphism $\tau(f)(y_1,y_2) = f(y_2,y_1)$ and $d\in
\End(k[y_1^{\pm},y_2^{\pm}])$ is defined by $d(f) =
-(q-q^{-1})\tss y_2\tss (y_1-y_2)^{-1}(f-\tau(f))$.

\subsection{Fiber functors and quantum affine algebras}

Let us compose the monoidal functor $F:\wt\caH(q)\to\caH(q)$ with the
monoidal fiber functor $F_N:\caH(q)\to\Vect$
considered in Sec.~\ref{subsec:ffhe}.
We get a fiber functor
$F_N:\wt\caH(q)\to\Vect$ which we denote by the same symbol.
By Corollary~\ref{fahc}, $F_N$ is
determined by its values on the generating object $X\in\wt\caH(q)$
together with its values on the generating morphisms $x\in
\End_{\wt\caH(q)}(X)$ and $c\in \End_{\wt\caH(q)}(X^{\otimes 2})$.
The object
$F_N(X)$ is the tensor product $V\otimes k[y^{\pm 1}]$,
where $V=k^N$. The endomorphism
$F_N(x)$ is the morphism, induced by the
multiplication by $y$ in $k[y^{\pm 1}]$,
\begin{equation}
F_N(x)(v\ot y^s) = v\ot y^{s+1}.
\end{equation}
The map $F_N(c):F_N(X)\otimes F_N(X)\to F_N(X)\otimes F_N(X)$ is
given as follows. Identifying $F_N(X)\otimes F_N(X)$ with $V\otimes
V\otimes k[y_1^{\pm 1},y_2^{\pm 1}]$ by
\ben
(v\ot y^r)\otimes (w\ot y^s)\mapsto
v\otimes w\otimes y_1^ry_2^s,
\een
we can write $F_N(c)$ as
\beq
F_N(c)\big(v\otimes
w\otimes f(y_1,y_2)\big) = R(v\otimes w)\otimes f(y_2,y_1) -
(q-q^{-1})v\otimes w\otimes y_2\ts\partial\tss f(y_1,y_2),
\eeq
where the operator
$R$ is defined in (\ref{dj}), and $\partial$
is the divided difference operator
\eqref{divdiff} extended to Laurent polynomials.

We will now formulate an analogue of Proposition~\ref{prop:yangf},
where the role of the Yangian is played by the
{\it quantum affine algebra\/} $\U_q(\wh\gl_N)$
(with the trivial center charge), also known as
the {\it quantum loop algebra}. The role of the vector
representation is now played by the space $V\otimes
k[y^{\pm 1}]$. Explicit formulas for the action
of $\U_q(\wh\gl_N)$ on this space
are analogous to \eqref{yangact} and they can
be found in \cite{grv:qg}.

\begin{prop}
The fiber functor $F_N:\wt\caH(q)\to\Vect$ factors through the category
of representations $\U_q(\wh\gl_N)\da\Mod$ of the
quantum affine algebra
$$\xymatrix{ \wt\caH(q)
\ar[rr]^{F_N} \ar[rd]_{GRV_N} & & \Vect \\
& \U_q(\wh\gl_N)\da\Mod \ar[ru] }$$ via a monoidal functor
$GRV_N:\wt\caH(q)\longrightarrow \U_q(\wh\gl_N)\da\Mod$
and the forgetful functor\newline
$\U_q(\wh\gl_N)\da\Mod\longrightarrow\Vect$.
\end{prop}

The proof is quite similar to that of
Proposition~\ref{prop:yangf} and amounts to checking that
$F_N(x)$ and $F_N(c)$ are morphisms of
$\U_q(\wh\gl_N)$-modules. We
omit the details; see also \cite{grv:qg}.

We call $GRV_N$ the {\it
Ginzburg--Reshetikhin--Vasserot functor}, as the following
version of the Schur--Weyl duality for
the quantum loop algebras was proved in \cite{grv:qg}; see also
\cite{cp:qa}.

\begin{prop}
The functor $GRV_N:\wt\caH(q)\longrightarrow
\U_q(\wh\gl_N)\da\Mod$ is full.
\end{prop}

\section{Localizations and categorical actions}\label{app}
\setcounter{equation}{0}

\subsection{Localizations with respect to discriminants}

Here we discuss some applications of the universal properties of
the affine Hecke category $\wt\caH(q)$ and its degenerate version
$\caL$. We will regard
$\wt\caH(q)$ and $\caL$ as respective
quantum deformations of the affine
symmetric category $\AS$ and the semi-affine symmetric
category $\SAS$ and we will show that these deformations are
trivial away from some
discriminant-type loci.

To formulate the precise statement, let $\Delta\subset \Mor\AS$ be the
monoidally and multiplicatively closed set of morphisms generated by
$x\otimes 1-1\otimes x$; see Sec.~\ref{subsec:asc}.
In other words, for each $n$ we consider
the multiplicatively closed set of
morphisms generated by
\ben
\Delta_n\in k[y_1^{\pm
1},\dots,y_n^{\pm 1}]^{S_n}\subset k[y_1^{\pm 1},\dots,y_n^{\pm 1}]*S_n
= \End_\AS(X^{\otimes n}),
\een
where $\Delta_n =
\prod_{i\not=j}(y_i-y_j)$ is the {\it discriminant polynomial\/}.
Note that the algebra of symmetric polynomials
$k[y_1^{\pm 1},\dots,y_n^{\pm 1}]^{S_n}$ coincides with the center
of the endomorphism algebra
$\End_\AS(X^{\otimes n})$ so that $\Delta_n$ commutes with all
morphisms in $\AS$. Denote by $\AS[\Delta^{-1}]$ the category of
fractions with respect to $\Delta$; see e.g. \cite{gz:cf}
for the definition.
This category has the form
$\cC(A_*)$, where $A_*=\{A_n\ |\ n\geqslant 0\}$
is the multiplicative sequence of the localized
algebras $A_n =
\big(k[y_1^{\pm 1},\dots,y_n^{\pm1}]*S_n\big)[\Delta^{-1}_n]$.
Therefore, the category
$\AS[\Delta^{-1}]$ is monoidal. A similar argument shows that
the category $\SAS[\Delta^{-1}]$ is also monoidal.
Moreover, the localization functors
$\AS\to\AS[\Delta^{-1}]$ and $\SAS\to\SAS[\Delta^{-1}]$ are monoidal.

It is well known from \cite{bz:ir}
that the center of the affine Hecke algebra
$\wt \He_n(q)$ for generic $q$
coincides with the algebra of
symmetric Laurent polynomials
$k[y_1^{\pm 1},\dots,y_n^{\pm 1}]^{S_n}$, while
the center of the degenerate affine
Hecke algebra $\Lambda_n$ coincides with the algebra of
symmetric polynomials
$k[y_1,\dots,y_n]^{S_n}$. Therefore, it is unambiguous to
define the respective categories of fractions as
\ben
\wt\caH(q)[\Delta^{-1}] = \cC(\wt \He_*(q)[\Delta^{-1}_*])
\qquad\text{and}\qquad
\caL[\Delta^{-1}] = \cC(\Lambda_*[\Delta^{-1}_*]).
\een

\begin{prop}\label{disc}
The assignment $(X,x,c)\mapsto (X,x,\wt c\tss)$, where
\begin{equation}\label{ahc}
\wt c = (x\otimes 1-1\otimes x)^{-1}\big((q^{-1}-q)(1\otimes x) +
(q(x\otimes 1)-q^{-1}(1\otimes x))c\big)
\end{equation}
defines a monoidal functor $\wt\caH(q)[\Delta^{-1}]\to\AS[\Delta^{-1}]$.

Moreover, the assignment $(X,x,c)\mapsto (X,x,\wt c\tss)$,
where
\begin{equation}\label{dahc}
\wt c = (x\otimes 1-1\otimes x)^{-1}\big(1\otimes 1 + (x\otimes
1-1\otimes x - 1\otimes 1)c\big)
\end{equation}
defines a monoidal functor $\caL[\Delta^{-1}]\to\SAS[\Delta^{-1}]$.
\end{prop}

\begin{proof}
In the affine Hecke category case,
write $\wt c=a+(q-a)t$ as an element of
$k[x_1,x_2]*S_2$ with $t=c$ and
$a=(q^{-1}-q)\tss x_2\tss(x_1-x_2)^{-1}$, where $x_1=x\otimes 1,\
x_2=1\otimes x$. Now we verify the relations
\begin{align}\label{cqez}
(\ts\wt c-q)(\ts\wt c+q^{-1}) &= 0,\qquad
(x\otimes 1)\tss\wt c - \wt c\ts(1\otimes x) = 1,\\
\label{czco}
(\ts\wt c\otimes 1_X)(1_X\otimes \wt c\ts)(\ts\wt c\otimes 1_X)
{}&= (1_X\otimes \wt c\ts)(\ts\wt c\otimes 1_X)(1_X\otimes \wt c\ts)
\end{align}
by direct computations in $k[x_1,x_2]*S_2$ and
$k[x_1,x_2,x_3]*S_3$, respectively, where in the latter
case we interpret the variables as $x_1=x\ot 1\ot 1$,
$x_2=1\ot x\ot 1$ and $x_3=1\ot 1\ot x$.
Noting that
$t\tss a=(q-q^{-1}-a)\tss t$ we
get
\begin{multline}
(a-q+(q-a)\ts t)(a+q^{-1}+(q-a)\ts t)\\
{}= (a-q)(a+q^{-1})
+ (a-q)(q-a)\ts t - (q-a)(a-q)\ts t + (q-a)(q^{-1}+a)\ts t^2 = 0
\nonumber
\end{multline}
and
\begin{multline}
(a+(q-a)t)\ts x_1(a-(q-a)t) \\
{}= a^2x_1
+ (q-a)x_2(q-q^{-1}-a)\ts t + ax_1(q-a)\ts t - (q-a)(q^{-1}+a)\ts t^2 \\
{}=x_2 + a\big(a(x_1-x_2) - (q^{-1}-q)x_2\big)
+(q-a)\big(a(x_1-x_2) - (q^{-1}-q)x_2\big)\ts t = x_2,
\nonumber
\end{multline}
thus proving \eqref{cqez}.
To verify \eqref{czco}, note that the relation is equivalent to
\begin{multline}
(a_{12}+(q-a)\ts t_{12})(a_{23}+(q-a)
\ts t_{23})(a_{12}+(q-a)\ts t_{12})\\
{}= (a_{23}+(q-a)\ts t_{23})(a_{12}+(q-a)
\ts t_{12})(a_{23}+(q-a)\ts t_{23}),
\label{ccxc}
\end{multline}
where we used the notation
$a_{12}=a\otimes 1,\ a_{23} = 1\otimes a$ and
$a_{13} = t_1a_{23}t_1 = t_2a_{12}t_2$. Now
compare the coefficients of the elements of $S_3$
on both sides of \eqref{ccxc}. They obviously
equal for the elements $t_1t_2$, $t_2t_1$
and $t_1t_2t_1=t_2t_1t_2$, while
for $1$, $t_1$ and $t_2$ we need to check
the following relations, respectively:
\ben
\bal
a_{12}a_{23}a_{12} + (q-a_{12})a_{13}(q^{-1}+a_{12})
&= a_{23}a_{12}a_{23} + (q-a_{23})a_{13}(q^{-1}+a_{23}),\\
a_{12}a_{23}(q-a_{12}) &+ (q-a_{12})a_{13}(q-q^{-1}-a_{12})
= a_{23}(q-a_{12})\tss a_{13},\\
a_{12}(q-a_{23})a_{13} &= a_{23}a_{12}(q-a_{23})
+ (q-a_{23})a_{13}(q-q^{-1}-a_{23}).
\eal
\een
However, all of them follow from the identity
\ben
a_{12}a_{23} - a_{12}a_{13} - a_{13}a_{23}
+ (q-q^{-1})a_{13} = 0,
\een
which is verified directly
by substituting the expressions
for $a_{12}$, $a_{23}$ and $a_{13}$ in terms of
$x_1$, $x_2$ and $x_3$.

In the case of the
degenerate affine Hecke category, the argument is quite similar.
We write $\wt c=a+(1-a)\ts t$ as an element of
$k[x_1,x_2]*S_2$ with $t=c$ and $a=(x_1-x_2)^{-1}$, where
$x_1=x\otimes 1,\ x_2=1\otimes x$. Now the relations
\ben
\bal
\wt c^{\ts 2}=1,\qquad
(x\otimes 1)\ts\wt c - \wt c\ts(1\otimes x) &= 1,\\
(\ts\wt c\otimes 1_X)(1_X\otimes \wt c\ts)(\ts\wt c\otimes 1_X)
&= (1_X\otimes \wt c\ts)(\ts\wt c\otimes 1_X)(1_X\otimes \wt c\ts)
\eal
\een
are verified
directly by computations in $k[x_1,x_2]*S_2$ and
$k[x_1,x_2,x_3]*S_3$, respectively, exactly as
in the Hecke category case.
\end{proof}

Let $\Delta(q)\subset \Mor\AS$ be a monoidally
and multiplicatively closed
set of morphisms generated by $q(x\otimes 1)-q^{-1}(1\otimes x)$.
That is, for each $n$ we consider
the multiplicatively closed set of
morphisms generated by
\ben
\Delta(q)_n\in k[y_1^{\pm 1},\dots,y_n^{\pm 1}]^{S_n}\subset
k[y_1^{\pm 1},\dots,y_n^{\pm 1}]*S_n = \End_\AS(X^{\otimes n}),
\een
where
$\Delta(q)_n = \prod_{i\not=j}(qy_i-q^{-1}y_j)$ is the  {\it
quantum discriminant polynomial}. We can also consider
$\Delta_q$ as a set of morphisms of $\wt\caH(q)$.
Similarly, let $\wt\Delta\subset \Mor\SAS$ be
a monoidally and multiplicatively closed
set of morphisms generated by $x\otimes 1-1\otimes x-1\otimes 1$,
so that for each $n$ we consider
the multiplicatively closed set of
morphisms generated by
\ben
\wt\Delta_n\in k[y_1,\dots,y_n]^{S_n}\subset k[y_1^,\dots,y_n]*S_n
= \End_\SAS(X^{\otimes n}),
\een
where $\wt\Delta_n =
\prod_{i\not=j}(y_i-y_j-1)$ is the  {\it degenerate quantum
discriminant polynomial}. Similar to the above, we can consider
$\wt\Delta$ as a set of morphisms
of $\caL$.

The proof of the following is completely
analogous to and partly follows from the proof of Proposition
\ref{disc}.

\begin{prop}\label{qdisc}
The assignment $(X,x,c)\mapsto (X,x,\wt c)$, where
\begin{equation}\label{iahc}
\wt c = (q(x\otimes 1)-q^{-1}(1\otimes x))^{-1}\big((x\otimes
1-1\otimes x)\ts c + (q^{-1}-q)(1\otimes x)\big)
\end{equation}
defines a monoidal functor
$\AS[\Delta(q)^{-1}]\to\wt\caH(q)[\Delta(q)^{-1}]$.

Moreover, the assignment
$(X,x,c)\mapsto (X,x,\wt c)$, where
\begin{equation}\label{idahc}
\wt c = (x\otimes 1-1\otimes x-1\otimes 1)^{-1}\big((x\otimes
1-1\otimes x)\ts c + 1\otimes 1\big)
\end{equation}
defines a monoidal functor
$\SAS[\wt\Delta^{-1}]\to\caL[\wt\Delta^{-1}]$.
\end{prop}

Combining Propositions~\ref{disc} and \ref{qdisc}, we come to the
following theorem, where we let
$D(q)=\Delta\Delta(q)$ denote the set of compositions of morphisms from
$\Delta$ and $\Delta(q)$, and let
$D=\Delta\wt\Delta$ denote the set of compositions of morphisms from
$\Delta$ and $\wt\Delta$.

\begin{thm}\label{degdi}
The monoidal categories $\AS[D(q)^{-1}]$ and
$\wt\caH(q)[D(q)^{-1}]$ are equivalent.

Moreover, the
monoidal categories $\SAS[D^{-1}]$ and
$\caL[D^{-1}]$ are
equivalent.
\end{thm}

\begin{proof}
Both statements follow from the observation that each pair
of constructions (\ref{ahc}) and
(\ref{iahc}), as well as (\ref{dahc}) and (\ref{idahc}),
are inverse to each other.
\end{proof}

\begin{remark}
By the Galois theory,
$k[y_1,\dots,y_n][\Delta_n^{-1}]*S_n$ is isomorphic to
the algebra
$$
M_{n!}\big(k[y_1,\dots,y_n]^{S_n}[\Delta_n^{-1}]\big)
$$
of $n!\times n!$ matrices with coefficients in
the localization $k[y_1,\dots,y_n]^{S_n}[\Delta_n^{-1}]$
of the algebra of symmetric polynomials. Theorem
\ref{degdi} implies that the algebra
$\wt\He_n(q)[(\Delta_n\Delta(q)_n)^{-1}]$ is isomorphic to the matrix
algebra
$$
M_{n!}\big(k[y_1^{\pm 1},\dots,y_n^{\pm1}]^{S_n}
[(\Delta_n\Delta(q)_n)^{-1}]\big).
$$
Similarly, the algebra
$\Lambda_n[(\Delta_n\wt\Delta_n)^{-1}]$ is isomorphic to the
matrix algebra
$$
M_{n!}\big(k[y_1^{\pm 1},\dots,y_n^{\pm
1}]^{S_n}[(\Delta_n\wt\Delta_n)^{-1}]\big).
$$
\end{remark}

\subsection{Orellana--Ram and Cherednik--Arakawa--Suzuki functors}

Here we give a construction, turning a braided monoidal category
with a Hecke object into a module category over the affine Hecke
category; see e.g. \cite{jk}, \cite{qi:ha}
for the definition of a module category.
Our construction has been motivated
by the work of Orellana and Ram~\cite{or:ab}.

Let $\cC$ be a braided monoidal category and let $c^{}_{X,Y}$
denote the braiding
\ben
c^{}_{X,Y}:X\ot Y\to Y\ot X.
\een
Fix an object $X$ of
$\cC$ and let $O:\cC\to\cC$ be the functor
of tensoring by $X$ from the
right: $O(Y) = O_X(Y) = Y\otimes X$. Note that $O$, as an object of
the monoidal category of endofuctors $\Funct(\cC,\cC)$, possesses a
Yang--Baxter operator $c$, which is (as a morphism in the functor
category $\Funct(\cC,\cC)$) the natural transformation
$O_{c^{}_{X,X}}$:
$$\xymatrix{O\circ O(Y) = Y\otimes X^{\otimes 2}
\ar[rr]^{1_Y\ot c^{}_{X,X}} && Y\otimes X^{\otimes 2} = O\circ O(Y)}.$$
Define
an endomorphism $x:O\to O$ (a natural transformation) as the
composition:
$$
\xymatrix{O(Y) = Y\otimes X \ar[rr]^(.6){c^{}_{Y,X}} &&
X\otimes Y \ar[rr]^(.4){c^{}_{X,Y}} && Y\otimes X = O(Y).}
$$

\begin{lem}
The triple $(O,x,c)$ defines a monoidal functor
$OR_X:\wt\cB\to\Funct(\cC,\cC)$.
\end{lem}
\begin{proof}
The following commutative diagram (which is a joint of two coherence
diagrams for the braiding)
guarantees that $c$ and $x$ satisfy the condition
(\ref{ah}):
$$\xymatrix{Y\otimes X^{\otimes 2}
\ar[d]^{1_Y\ot c^{}_{X,X}} \ar[rr]^{c^{}_{Y\ot X,X}} &&
X\otimes Y\otimes X
\ar[rr]^{c^{}_{X,Y\ot X}}
\ar@{=}[d]&& Y\otimes X^{\otimes 2} \\
Y\otimes X^{\otimes 2} \ar[rr]^{c^{}_{Y\ot X,X}} && X\otimes Y\otimes X
\ar[rr]^{c^{}_{X,Y\ot X}} && Y\otimes X^{\otimes 2}
\ar[u]^{1_Y\ot c^{}_{X,X}} }
$$
thus proving the claim.
\end{proof}

We call $OR_X$ the {\it Orellana--Ram}
functor corresponding to $X\in\cC$.
In the particular case
$\cC=\cB$ we get a monoidal functor
$OR:\wt\cB\to\Funct(\cB,\cB)$ corresponding to the generating
object of $\cB$. It is easy to see that this functor is faithful
(injective on morphisms).

An object $X$ of a braided monoidal category $\cC$
will be called a {\it
Hecke object\/}, if the braiding
$c_{X,X}\in \End_\cC(X^{\otimes 2})$ satisfies
the equation
$(c_{X,X}-q\ts1_{X\ot X})(c_{X,X}+q^{-1}\ts1_{X\ot X})=0$
for some non-zero scalar $q\in k$.

\begin{prop}
If $X$ is a Hecke object of $\cC$, then the functor $OR_X$ factors
through the affine Hecke category $\wt\caH(q)$, giving rise to a
functor $OR_X:\wt\caH(q)\to\Funct(\cC,\cC)$.
\end{prop}
\begin{proof}
By definition, the Yang--Baxter operator $c$ on
the functor $O$ satisfies
the equation $(c-q)(c+q^{-1})=0$.
\end{proof}

Now we describe a degenerate analogue of the Orellana--Ram functors. The
construction (a special case of
which was studied in \cite{as:lad} and \cite{ch:mr})
requires an infinitesimal version of the notion of braided category.
The notion of {\it chorded categories\/} was virtually defined
by Drinfeld~\cite{d:qq} and was studied in \cite{ka:qg} under
the name {\it infinitesimal symmetric categories}.
Due to their relation with
Kontsevich's chord diagrams we will call them chorded categories.

Let $\cC$ be a symmetric monoidal category with
the symmetry $c_{X,Y}$.
A {\it chording} on $\cC$ is a
natural collection of morphisms $h_{X,Y}:X\otimes Y\to X\otimes Y$
satisfying the conditions:
\begin{equation}\label{symcho}
c_{X,Y}h_{Y,X} = h_{X,Y}c_{X,Y}
\end{equation}
and
\begin{equation}\label{cho}
h_{X,Y\otimes Z} = h_{X,Y}\otimes 1_Z + (1_X\otimes
c_{Y,Z})(h_{X,Z}\otimes 1_Y) (1_X\otimes c_{Y,Z})^{-1}.
\end{equation}
A symmetric monoidal category with a chording will be called a
{\it chorded
category}.

A symmetric monoidal functor $F:\cC\to\caD$
between chorded categories is {\it
chorded} if
\ben
F_{X,Y}F(h_{X,Y}) = h_{F(X),F(Y)}F_{X,Y},
\een
where
\ben
F_{X,Y}:F(X\otimes Y)\to F(X)\otimes F(Y)
\een
is the monoidal
constraint of $F$.

Now  we describe a construction, which endows the category
of representations of a Lie algebra with a chorded structure.
A Lie
algebra $\mathfrak g$ will be called
a {\it Casimir Lie algebra}, if it is equipped with a
symmetric and $\mathfrak g$-invariant element
$\Omega\in\mathfrak g\ot \mathfrak g$.

\bpr
Let $(\mathfrak g,\Omega)$ be Casimir Lie algebra
and let $M,N\in \Rep(\mathfrak g)$.
The
relation
\ben
h_{M,N}(m\otimes n) = \Omega(m\otimes n),
\qquad m\in M,\quad n\in N,
\een
defines a chorded structure on the category of representations
$\Rep(\mathfrak g)$ of $\mathfrak g$.
\epr

\begin{proof}
The $\mathfrak g$-invariance of $\Omega$ implies that $h_{M,N}$ is a
homomorphism of representations of $\mathfrak g$; i.e.,
a morphism in $\Rep(\mathfrak g)$.
The condition \eqref{symcho} follows from the symmetry property
of $\Omega$. Finally, we have the identity
$$(1\otimes\Delta)(\Omega) = \Omega_{12} + \Omega_{13},$$
where $\Delta(y)=y\ot 1+1\ot y$ for $y\in\mathfrak g$,
which
verifies the condition \eqref{cho}.
\end{proof}

An object $X$ of a chorded category $\cC$
will be called a {\it degenerate
Hecke object}, if the chording is proportional
to the commutativity morphism,
$h_{X,X}=\lambda\ts c_{X,X}$ for some non-zero
scalar $\lambda\in k$.
Let $\cC$ be a chorded monoidal category and let $X$ be a degenerate
Hecke object of $\cC$.
Let $O:\cC\to\cC$ be
a functor of tensoring by $X$ from the right, $O(Y) = O_X(Y) =
Y\otimes X$. Note that $O$, as an object of the monoidal category of
endofuctors $\Funct(\cC,\cC)$, possesses a Yang--Baxter operator $c$,
which is (as a morphism in the functor category $\Funct(\cC,\cC)$) the
natural transformation $O_{c^{}_{X,X}}$:
$$\xymatrix{O\circ O(Y)
= Y\otimes X^{\otimes 2}\ar[rr]^{1_Y\ot c^{}_{X,X}} &&
Y\otimes X^{\otimes 2}
= O\circ O(Y)}.$$
Define an endomorphism $x:O\to O$ (a natural
transformation):
$$
\xymatrix{O(Y) = Y\otimes X
\ar[rr]^{\lambda^{-1}h_{Y,X}} && Y\otimes X = O(Y).}
$$

\begin{prop}
Let $\cC$ be a chorded monoidal category and let $X$ be a degenerate
Hecke object of $\cC$.
Then the triple $(O,x,c)$ defines a monoidal functor
$CAS_X:\caL\to\Funct(\cC,\cC)$.
\end{prop}

\begin{proof}
The condition (\ref{dege}) follows from the chording axiom. Indeed,
the natural transformation $(x\otimes 1)\ts c-c\ts (1\otimes x)\in
\End\big(O(X)\circ O(X)\big)$ evaluated at $Y\in\cC$ has the form
\ben
\bal
\lambda^{-1}(h_{Y,X}\otimes 1_X)(1_Y\otimes c_{X,X})
&- \lambda^{-1}(1_Y\otimes c_{X,X})h_{Y\otimes X,X}\\
{}=\lambda^{-1}(h_{Y,X}&\otimes 1_X)(1_Y\otimes c_{X,X})
- \lambda^{-1}(1_Y\otimes c_{X,X})\\
{}&\times\Big(1_Y\otimes h_{X,X} + (1_Y\otimes
c_{X,X})(h_{Y,X}\otimes 1_X)(1_Y\otimes c_{X,X})\Big)\\
& \qquad\qquad\qquad\qquad=
\lambda^{-1}(1_Y\otimes c_{X,X})(1_Y\otimes h_{X,X}) = 1_{Y\ot X\ot X},
\eal
\een
as required.
\end{proof}

We call the functor $CAS_X$ the {\it
Cherednik--Arakawa--Suzuki} functor corresponding to $X\in\cC$;
see the following examples.

\begin{example}
{\it Representations of $\gl_N$}.
Let $V$ be the $N$-dimensional vector representation of the Lie
algebra $\gl_N$. Consider the chorded structure on the category
$\cC=\Rep(\gl_N)$ of
representations of $\gl_N$
given by the standard Casimir
element $C\in \gl_N\ot\gl_N$.
Then $V$ is a degenerate Hecke object of $\cC$.
This gives a functor $CAS_V:\caL\to\Funct(\cC,\cC)$ studied in some
form in \cite{as:lad} and \cite{ch:mr}; see also
\cite[Lemma 7.21]{cr:de}. Note also its relationship
with the functors used in \cite{bs:hw3}
(for $\gl_{m+n}$-modules) and in \cite{bs:hw4}
(for $\gl_{m|n}$-modules) to establish
equivalences of categories, giving
rise to explicit diagrammatical descriptions of parabolic highest
weight categories for $\gl_{m+n}$ and for the category of finite
dimensional modules over $\gl_{m|n}$.
\qed
\end{example}

\begin{example}
{\it The category $\cS$}.
The category $\cS$ has a chorded structure uniquely defined by
$h_{X,X}=c_{X,X}$. In particular, the generator $X$ is a degenerate
Hecke object of $\cS$. Thus we get a monoidal functor
$CAS_X:\caL\to\Funct(\cS,\cS)$.
\qed
\end{example}

\end{document}